\documentclass[11pt]{article}
\usepackage[french,english]{babel}
\usepackage{graphicx}
\usepackage{graphics}
\usepackage{amsfonts}
\usepackage{amscd}
\usepackage{amsthm}
\usepackage{amsmath}
\usepackage{indentfirst}
\usepackage[all]{xy}
\usepackage{amssymb, amsmath, amsthm, amsgen, amstext, amsbsy, amsopn}
\usepackage{epic,eepic}
\usepackage{hyperref}
\usepackage{enumitem}
\usepackage{color}
\usepackage{dsfont}
\usepackage{comment}
\usepackage{esint}
\usepackage[margin=1.3in]{geometry}

\usepackage{eufrak}

\usepackage{appendix}

\DeclareMathOperator{\Ham}{Ham}

\theoremstyle{plain}
\newtheorem*{thm*}{Theorem}
\newtheorem*{cor*}{Corollary}
\newtheorem{thm}{Theorem}
\newtheorem{lemma}{Lemma}[section]
\newtheorem{prop}[lemma]{Proposition}
\newtheorem{claim}[lemma]{Claim}
\newtheorem{cor}[lemma]{Corollary}
\newtheorem{theorem}[lemma]{Theorem}

\theoremstyle{definition}

\newtheorem{rem}[lemma]{Remark}

\newtheorem*{thmA}{Theorem A}
\newtheorem*{thmB}{Theorem B}

\def\supp{\text{supp}}
\def\vol{\text{Vol}}

\title{\textbf{Flexibility of the adjoint action of the group of Hamiltonian diffeomorphisms}}
\author{\textbf{Lev Buhovsky, Maksim Stoki\'c}}

\begin{document}

\maketitle

\begin{abstract}
On a closed and connected symplectic manifold, the group of Hamiltonian diffeomorphisms has the structure of an infinite-dimensional Fr\'echet Lie group, where the Lie algebra is naturally identified with the space of smooth and zero-mean normalized functions, and the adjoint action is given by pullbacks. We show that this action is flexible: for every non-zero smooth and zero-mean normalized function $ u $, any other smooth and zero-mean function $ f $ can be written as a finite sum of elements in the orbit of $u$ under the adjoint action. Additionally, the number of elements in this sum is dominated by the uniform norm of $f$. This result can be interpreted as a (bounded) infinitesimal version of Banyaga's theorem on the simplicity of the group of Hamiltonian diffeomorphisms.  
\end{abstract}

\selectlanguage{french}
\begin{center}
    \LARGE\textbf{Flexibilit\'e de l'action adjointe du groupe des diff\'eomorphismes hamiltoniens}
\end{center}
\begin{abstract}
Sur une vari\'et\'e symplectique compacte et connexe, le groupe des diff\'eomorphismes hamiltoniens poss\`ede la structure d'un groupe de Lie de Fr\'echet de dimension infinie, dont l'alg\`ebre de Lie s'identifie naturellement \`a l'espace des fonctions lisses normalis\'ees de moyenne nulle, et l'action adjointe est par tir\'es en arri\`ere. Nous d\'emontrons que cette action est flexible~: pour chaque fonction lisse non nulle, normalis\'ee et de moyenne nulle $u$, toute autre fonction lisse, et de moyenne nulle $f$ peut \^etre \'ecrite comme une somme finie d'\'el\'ements de l'orbite de $u$ sous l'action adjointe. De plus, le nombre d'\'el\'ements dans cette somme est domin\'e par la norme uniforme de $f$. Ce r\'esultat peut \^etre interpr\'et\'e comme une version infinit\'esimale (born\'ee) du th\'eor\`eme de Banyaga sur la simplicit\'e du groupe des diff\'eomorphismes hamiltoniens. 
\end{abstract}
\selectlanguage{english}

\section{Introduction and main results}

Consider a closed and connected symplectic manifold $(M,\omega)$ of dimension $2n$. Let $C^{\infty}_0(M)$ be the space of smooth functions that are zero-mean normalized with respect to the volume form $\omega^n$. When equipped with the $C^{\infty}$-topology, the group of Hamiltonian diffeomorphisms $\Ham(M,\omega)$ is an infinite-dimensional Fr\'echet Lie group, whose Lie algebra $\mathcal{A}$ can be identified with the space $C^{\infty}_0(M)$. The adjoint action of $\Ham(M,\omega)$ on $\mathcal{A}$ is given by $\mathrm{Ad}_{\phi}f=f\circ\phi^{-1}$, for every $f\in\mathcal{A}=C^{\infty}_0(M)$ and $\phi\in\Ham(M,\omega)$. Our main result shows flexibility of the adjoint action in the following sense:

\begin{thm}\label{ThmMain}
    Let $(M,\omega)$ be a closed and connected symplectic manifold, and let $u\in C_0^{\infty}(M)$ be a non-zero function. There exists $N\in\mathbb{N}$ that only depends on $u$, such that for any $f\in C^{\infty}_0(M)$ with $\|f\|_{\infty}\leq 1$, one can write
    \[f=\sum_{i=1}^N\Phi_{i}^*u,\]
    for some Hamiltonian diffeomorphisms $\Phi_{i}\in\Ham(M,\omega)$.
\end{thm}

For a non-zero function $ u \in C^\infty_0(M) $, denote by $ N(u) \in \mathbb{N} $ the {\em minimal} $ N $ (the number of summands) as in the theorem. Then $ N(u) $ is invariant under the action of symplectic diffeomorphisms, and it would be interesting to understand its properties better.
 
One can view Theorem \ref{ThmMain} as an infinitesimal analogue of the simplicity of the group $ \Ham(M,\omega) $ proved by Banyaga \cite{Ba1}. Indeed, simplicity of $ \Ham(M,\omega) $ is equivalent to saying that for any non-trivial $ \Phi \in \Ham(M,\omega) $, any other $ \Psi \in \Ham(M,\omega) $ can be written as
\begin{equation}\label{BanyagaEq}
    \Psi = \prod_{i=1}^m \Theta_i^{-1} \Phi^{\pm 1} \Theta_i .
\end{equation}
where $ \Theta_i \in \Ham(M,\omega) $ for each $ i $. Now, if we fix an autonomous Hamiltonian function $ H $ and assume that $ \Phi^\varepsilon $ is the time-$\varepsilon$ map of the Hamiltonian flow of $ H $ (when $ \varepsilon $ is small), then up to $ o(\varepsilon) $ the Hamiltonian diffeomorphism $ \Psi^\varepsilon = \prod_{i=1}^m \Theta_i^{-1} (\Phi^{\varepsilon})^{\pm 1} \Theta_i $ equals to the time-$\varepsilon$ map of the Hamiltonian flow 
generated by $ F = \sum_{i=1}^m \pm \Theta_{i}^*H $. Note that in Theorem \ref{ThmMain} subtraction is not needed, and only addition is used for representing the function $ f $ via pullbacks of $ u $ by Hamiltonian diffeomorphisms. Moreover, Theorem \ref{ThmMain} guarantees existence of such a representation where the number of terms does not depend on the function $ f $ (provided that $ \| f \|_\infty \leq 1 $). 

\begin{rem}
    Comparing our result with Banyaga's simplicity theorem, one may ask if the number of terms $m$ in the representation (\ref{BanyagaEq}) is bounded from above by some number $m_0$, provided that $\|\Psi\|_{\mathrm{Hofer}}\leq 1$. The assumption on the Hofer's norm of $\Psi$ is necessary, since the triangle inequality implies $\|\Psi\|_{\mathrm{Hofer}}\leq m\cdot\|\Phi\|_{\mathrm{Hofer}}$. However, this assumption is not enough to bound $m$. A function $r:\Ham(M,\omega)\rightarrow\mathbb{R}$ is called a \textit{homogeneous quasimorphism} if there exists a constant $D$ such that $|r(\phi\circ\psi)-r(\phi)-r(\psi)|\leq D$ and $r(\phi^m)=mr(\phi)$ for all $\phi,\psi\in\Ham(M,\omega)$ and $m\in\mathbb{Z}$. If a homogeneous quasimorphism on $\Ham(M,\omega)$ exists, and $m$ is bounded by $m_0$, we get
    \[\|\Psi\|_{\mathrm{Hofer}}\leq 1\implies |r(\Psi)|\leq C := (m_0-1)D+m_0\cdot|r(\Phi)|.\]
    Now let $\Psi\in\Ham(M,\omega)$ with $\|\Psi\|_{\mathrm{Hofer}}\leq 1/2$, and let $k=\lfloor 1/\|\Psi\|_{\mathrm{Hofer}}\rfloor$. Note that $\|\Psi^k\|_{\mathrm{Hofer}}\leq k\cdot\|\Psi\|_{\mathrm{Hofer}}\leq 1$, therefore we can apply the previous estimate for $\Psi^k$ and get
    \[k\cdot|r(\Psi)|=|r(\Psi^k)|\leq C.\]
    From here we conclude
    \[\|\Psi\|_{\mathrm{Hofer}}\leq 1/2\implies |r(\Psi)|\leq2C \|\Psi\|_{\mathrm{Hofer}}.\]
    In particular, we get the Hofer continuity of $r$ at the identity. However, there are examples of symplectic manifolds and homogeneous quasimorphisms that are not Hofer continuous at the identity (see \cite{Kh} for examples).
\end{rem}

\begin{rem}
Despite the simplicity of $ \Ham(M,\omega) $ and Theorem \ref{ThmMain}, the Lie algebra $\mathcal{A}$ is not simple (this contrasts with the case of finite-dimensional Lie groups). Indeed, for any open subset $ U \subset M $, the space of all $ f \in C^{\infty}_0(M) $ that vanish on $ U $ 
forms an ideal of $ \mathcal{A} $ which is nontrivial and proper provided that $ U \neq \emptyset $ and that $ M\setminus U $ has a non-empty interior.
\end{rem}

In \cite{BO} Buhovsky and Ostrover showed the following result, which was also recently reproved by Lempert \cite{Le} via an elegant functional analytic approach:

\begin{thmA}[Buhovsky-Ostrover \cite{BO}; Lempert \cite{Le}] \label{ThmBO} 
Let $(M,\omega)$ be a closed symplectic manifold.
Any $\Ham(M,\omega)$-invariant pseudo-norm $\| \cdot \|$ on
${\mathcal A}$ that is continuous in the $C^{\infty}$-topology, is
dominated from above by the $L_{\infty}$-norm i.e., $\| \cdot \|
\leq C \| \cdot \|_{\infty}$ for some constant $C$.
\end{thmA}

\noindent Theorem \ref{ThmMain} readily implies that one can remove the condition of continuity in $C^{\infty}$ topology:
\begin{thm}
Let $ (M,\omega) $ be a closed symplectic manifold. Any $\Ham(M,\omega)$-invariant norm on the space $ {\cal A} = C^{\infty}_0(M)$ is bounded from above by a constant multiple of the $L_{\infty}$-norm.
\end{thm}

\begin{proof}
    Without loss of generality we can assume that $ M $ is connected. Fix a non-zero function $u\in C^{\infty}_0(M)$, and let $f\in C^{\infty}_0(M)$. Theorem \ref{ThmMain} gives us a representation
    \[f=\|f\|_{\infty}\cdot\sum_{i=1}^N\Phi^*_iu,\]
    where $N$ depends only on $u$. Applying the triangle inequality and $\Ham(M,\omega)$-invariance of the norm $\|\cdot\|$ we get $\|f\|\leq N\cdot\|u\|\cdot \|f\|_{\infty}$.
\end{proof}
    
One ingredient in the proof of Theorem \ref{ThmMain} is a property of solutions of the equation $ h(t) = f(t+\alpha)-f(t) $, where $ h : S^1 \rightarrow \mathbb{R} $ is a given smooth function of zero mean, $ \alpha \in \mathbb{R} $ is an irrational number which is badly approximable by rationals, and $ f : S^1 \rightarrow \mathbb{R} $ is an unknown function whose properties are of interest (see Lemma \ref{rotation-Lemma} in Section \ref{subsec:AuxLemmas} below). We remark that this equation is a simplest example of a small denominators problem \cite{A}, and its study already appeared in Hilbert's book \cite[Chapter 17, Section 5]{Hi}. Equations of that type also play an important role in establishing simplicity of classical diffeomorphisms groups \cite{Ba2}, however, interestingly enough, the way the equation is used in our proof of Theorem \ref{ThmMain}, seems to be different.

In \cite{P}, Polterovich proved the following remarkable averaging property:
\begin{thmB}[Polterovich \cite{P}]
Let $ (M,\omega) $ be a closed and connected symplectic manifold, and let $ u \in C(M) $ be a continuous function which is zero-mean normalized with respect to the volume form $\omega^n$. Then for every $ \varepsilon > 0 $ there exist $ \Phi_1, \ldots, \Phi_N \in \Ham(M,\omega) $ such that 
$$ \frac{1}{N} | u \circ \Phi_1(x) + \cdots + u \circ \Phi_N(x) | < \varepsilon $$
for every $ x \in M $.
\end{thmB}
In the case of an open connected symplectic manifold such a statement holds as well (for compactly supported $ u $) and plays an important role in the proof of Theorem \ref{ThmMain}. But in fact, for a {\em closed} connected symplectic manifold $ (M,\omega) $ and a {\em smooth} non-zero function $ u \in C^\infty_0(M) $, Theorem \ref{ThmMain} implies a sharp version of Theorem C, since in particular it shows that {\em the zero function can be represented as such a sum of pullbacks of $ u $ by Hamiltonian diffeomorphisms}. It is natural to ask whether this sharp version of Theorem C (or more generally, a version of Theorem \ref{ThmMain}) also holds for {\em continuous} zero-mean normalized functions $ u $ if we allow to use pullbacks by {\em Hamiltonian homeomorphisms}\footnote{This notion was initially introduced by M\"uller and Oh who defined a Hamiltonian homeomorphism to be the time-$1$ map of a continuous/topological Hamiltonian flow, see \cite{OM} for more details (cf. also \cite{Mu}). A different version of this notion was considered in \cite{L-C,BHS}, where a Hamiltonian homeomorphism was defined to be a homeomorphism that can be obtained as a uniform limit of a sequence of Hamiltonian diffeomorphisms.}. However, here one can produce a counterexample, see Proposition \ref{rem:counter-examp} below. It would be interesting to verify whether this is nevertheless true if one imposes certain additional regularity assumptions on $ u $ (for instance, we can restrict to the case when on some open and connected subset of $M$ the function $u$ is smooth and non-constant).
\begin{prop} \label{rem:counter-examp}
On a closed symplectic manifold $ (M,\omega) $, let us construct a continuous and zero-mean normalized function $ u : M \rightarrow \mathbb{R} $ such that for no collection of volume preserving homeomorphisms (in particular, Hamiltonian homeomorphisms)  $ \Phi_1, \ldots, \Phi_k : M \rightarrow M $, we have 
\begin{equation} \label{eq:sum-pullbacks-non-zero}
\sum_{i=1}^k \Phi_i^*u \equiv 0 .
\end{equation}
\end{prop}
\begin{proof}
First, fix a smooth chart $ (\varphi,V) $ on $ M $, such that $ 0 \in B \subset \varphi(V) \subset \mathbb{R}^{2n} $, where $ B \subset \mathbb{R}^{2n} $ is the closed unit ball centered at $ 0 $. Furthermore, let $ S \subset \mathbb{R} $ be a dense subset with the property that for any $ t_1,\ldots,t_m \in S $ we have $ t_1 + \ldots + t_m \neq 0 $ (for instance we can take $ S = ((-\infty,0) \cap \mathbb{Q}) \cup ((0,+\infty) \cap \sqrt{2}\mathbb{Q}) $). Now take some $ a,b \in S $, $ a > 0 > b $ such that 
$ a \cdot \vol(\varphi^{-1}(B)) + b \cdot \vol(M\setminus \varphi^{-1}(B)) > 0 $, and then by using a Cantor-like procedure find a continuous non-increasing function $ h : [0,1] \rightarrow \mathbb{R} $ such that $ h(0) = a $, $ h(1) = b $, and such that $ h $ is locally constant on an open set $ W \subset (0,1) $, where the measure of $ [0,1] \setminus W $ is zero and $ h(W) \subset S $. For every $ \alpha \in (0,+\infty) $ define the function $ u_\alpha \in C(M) $ by $ u_{\alpha}(x) = h(|\varphi(x)|^\alpha) $ for $ x \in \varphi^{-1}(B) $ and $ u_{\alpha}(x) = b $ for $ x \in M \setminus \varphi^{-1}(B) $. Since $ b < 0 $ and $ a \cdot \vol(\varphi^{-1}(B)) + b \cdot \vol(M\setminus \varphi^{-1}(B)) > 0 $, by the intermediate value theorem there exists $ \alpha \in (0,+\infty) $ such that $ u = u_\alpha $ is mean-normalized. Moreover this function $ u $ is locally constant on an open subset $ U \subset M $ with $ \vol(U) = \vol(M) $, such that moreover $ u(U) \subset S $. In view of our choice of $ S $ (sum of any finite collection of elements of $ S $ is non-zero), this immediately implies $(\ref{eq:sum-pullbacks-non-zero})$.\end{proof}

\subsection{Finite-dimensional analogue of Theorem \ref{ThmMain}}

It is natural to ask whether the statement of Theorem \ref{ThmMain} holds for finite-dimensional simple Lie algebras. Let $\mathfrak{g}$ be the Lie algebra of a finite-dimensional simple Lie group $G$. For a non-zero element $\xi\in\mathfrak{g}$, we define $\mathcal{V}_{\xi}=\{\mathrm{Ad}_{g}\,\xi\mid g\in G\}$. For each $m\in\mathbb{N}$, we define
\[\mathcal{V}_{\xi,m}=\underbrace{\mathcal{V}_{\xi}+\ldots+\mathcal{V}_{\xi}}_{m}.\]

\noindent\textbf{Question.} \textit{Does there exist $m$, such that $\mathcal{V}_{\xi,m}$ contains an open neighbourhood of $0$ in $\mathfrak{g}$?}\\

\noindent The symplectic linear group $\mathrm{Sp}(2n,\mathbb{R})$ serves as an example of a non-compact Lie group, whose Lie algebra $\mathfrak{sp}(2n,\mathbb{R})$ is simple and does not satisfy the above property. Recall that a matrix $H\in\mathfrak{sp}(2n,\mathbb{R})$ if and only if $JH$ is symmetric, where $J$ is the standard skew-symmetric matrix. Therefore, we can identify $\mathfrak{sp}(2n,\mathbb{R})$ with quadratic forms on $\mathbb{R}^{2n}$, where the adjoint action is given by pullbacks. If we consider an element of $\mathfrak{sp}(2n,\mathbb{R})$ corresponding to a positive quadratic form, then the sum of any number of its adjoints by elements of $\mathrm{Sp}(2n,\mathbb{R})$ will also yield a positive quadratic form. Consequently, one cannot generate a neighborhood of $0$ in $\mathfrak{sp}(2n,\mathbb{R})$ by sum of adjoints. On the other hand, the following theorem shows that the answer is positive when $G$ is compact.

\begin{thm}\label{FiniteDimThm}
    Let $\mathfrak{g}$ be the Lie algebra of a compact finite-dimensional simple Lie group $G$. For every non-zero element $\xi\in\mathfrak{g}$ there exists $m\in\mathbb{N}$ such that $\mathcal{V}_{\xi,m}$ contains an open neighbourhood of $0$ in $\mathfrak{g}$.
\end{thm}
The proof of the theorem is given in the Section \ref{FinieDimSection}.
  
\subsection{Sketch of the proof of Theorem \ref{ThmMain}}

\noindent The proof of Theorem \ref{ThmMain} consists of three parts, which we briefly explain:\\

\noindent\textbf{The microscale statement:} In this step we prove a weaker microscale version of the theorem, where in particular we impose a bound on the third derivative of $f$. Moreover, we represent $f$ as a sum $\sum\pm\Phi_i^{*}u$, with half of the signs plus, half minus. This condition appears in the argument involving Fourier series (Lemma \ref{rotation-Lemma}). See Section \ref{Section:Microscale}. \\

\noindent\textbf{The local statement:} We start with the normalized smooth function $f$ supported inside the unit $2n$-dimensional cube, assuming that $ \| f \|_\infty \leq 1 $. In order to apply the microscale statement we need the bound on the third derivative of $f$. We cover the unit cube by cubes of small enough size $\varepsilon$. Roughly speaking, we write $f=\sum_{i=1}^m f_i$, where $m$ depends only on the dimension (and not on $\varepsilon$) and each $f_i$ is the sum of normalized functions supported in disjoint cubes of size $\varepsilon$, and moreover satisfy the bound on the third derivative (appearing in the microscale statement). Then we can apply the microscale statement to each $f_i$. See Section \ref{Section:Local}. \\

\noindent\textbf{Localization:} We cover the manifold by Darboux charts and show that any normalized smooth function can be written as a sum of normalized smooth functions with supports in the Darboux charts, with estimates on the uniform norms of the summands. Moreover, for every chart we find an appropriate pair of Hamiltonian diffeomorphisms which composed with $u$ give plus/minus the first coordinate function in that chart, up to an additive constant. Then we deduce Theorem \ref{ThmMain} by applying an averaging property from \cite{P} combined with the local statement. See Section \ref{sec:proofThmMain}.

\subsection{Notation} For an open set $ W \subset \mathbb{R}^d $ and a smooth function $f:W \rightarrow\mathbb{R}$ we define
\[\|D^r f\|_{\infty}=\max_{i_1+\cdots+i_d=r}\sup_{W}\Big|\frac{\partial^r f}{(\partial x_1)^{i_1}(\partial x_2)^{i_2}\cdots(\partial x_d)^{i_d}}\Big|.\]
On the standard symplectic space $ (\mathbb{R}^{2n},\omega) $, we will denote coordinates by $ (x_1,x_2,\ldots, x_{2n}) $, so that 
$ \omega = dx_1 \wedge dx_2 + \cdots + dx_{2n-1} \wedge dx_{2n} $. Additionally, we will denote by $\phi^t_H$ a Hamiltonian flow generated by a Hamiltonian function $H$. The space of compactly supported Hamiltonian diffeomorphisms will be denoted $\Ham_{\mathrm{c}}(M,\omega)$.

\subsection{Acknowledgements} 
We thank Albert Fathi, Leonid Polterovich and Vukašin Stojisavljević for valuable and inspiring remarks. We also thank the anonymous referees for many valuable suggestions and comments. The authors were partially supported by ERC Starting Grant 757585 and ISF Grant 2026/17.

\section{A microscale version of the main result} \label{Section:Microscale}

\begin{theorem}[Microscale statement]\label{ThmMicro}
    Let $\varepsilon >0$ and $ b \in \mathbb{R} $, and let $v : (-4\varepsilon,4\varepsilon)^{2n} \rightarrow \mathbb{R} $ be the function given by 
    \[v(x_1,x_2,\ldots,x_{2n})= \frac{x_1}{\varepsilon}+b.\]
    There exists $C=C(n)>0$ such that for any smooth function $f : (-4\varepsilon,4\varepsilon)^{2n} \rightarrow \mathbb{R} $ supported in 
    $ (-\varepsilon,\varepsilon)^{2n} $ and which satisfies $\|D^3f\|_{\infty}\leq\frac{C}{\varepsilon^3}$ and $\int_{(-\varepsilon,\varepsilon)^{2n}}f\omega^n=0$, one can write
    \[f=\sum_{i=1}^{4n} \Phi_{i,+}^*v-\Phi_{i,-}^*v\]
    for some Hamiltonian diffeomorphisms $\Phi_{i,\pm}\in\Ham_{\mathrm{c}}((-4\varepsilon,4\varepsilon)^{2n})$.
\end{theorem}

The microscale statement is the essential ingredient for the proof of the local version of the main result. Its proof relies on the following two propositions.

\begin{prop}\label{MainProp1}
    Let $L>0$ and let $f : (-2L,2L)^{2n} \rightarrow \mathbb{R} $ be a smooth function supported in $ (-L,L)^{2n} $ such that $\int_{(-L,L)^{2n}}f\omega^{n}=0$. There exist $C_1 = C_1(n) >0$, a collection of smooth functions $g_1,\ldots,g_{2n}\in C^{\infty}_{\mathrm{c}}((-2L,2L)^{2n})$, and a collection of Hamiltonian diffeomorphisms $\Phi_1,\ldots,\Phi_{2n}\in\Ham_{\mathrm{c}}((-2L,2L)^{2n})$ such that
    \begin{enumerate}
        \item $f=\sum_{i=1}^{2n}\Phi_i^* g_i-g_i$,
        \item $(\forall\, 1\leq i\leq 2n)\,\|Dg_i\|_{\infty}\leq C_1L^2\|D^3f\|_{\infty}$.
    \end{enumerate}
\end{prop}

\begin{prop}\label{MainProp2}
  Let $L>0 $, $\lambda\neq 0$, and let $h: (-2L,2L)^{2n} \rightarrow \mathbb{R} $ be a smooth function supported in $ (-L,L)^{2n} $ such that $\|Dh\|_{\infty}\leq |\lambda|/8$. 
Define $ v: (-2L,2L)^{2n} \rightarrow \mathbb{R} $ by $v(x_1,\ldots,x_{2n})=\lambda x_1$. There exists a Hamiltonian diffeomorphism $\Phi\in\Ham_{\mathrm{c}}((-2L,2L)^{2n})$ such that the equality $\Phi^* v-v=h$ holds on $(-L,L)^{2n}$.
\end{prop}

\subsection{Auxiliary lemmas} \label{subsec:AuxLemmas}

The proofs of Propositions \ref{MainProp1} and \ref{MainProp2} rely on several lemmas which we state and prove in this section.

\begin{lemma}\label{PoincareLemma}
    Let $ d \geq 1 $ be an integer. For any $f\in C^{\infty}_{\mathrm{c}}((-1,1)^d)$ satisfying $\int_{(-1,1)^d}f(x)dx=0$, one can find $f_1,\ldots,f_d\in C^{\infty}_{\mathrm{c}}((-1,1)^d)$ with the following properties:
    \begin{enumerate}
        \item $f=\sum_{j=1}^d f_j$,
        \item $(\forall\, 1\leq j\leq d)$ and any $(x_1,\ldots,x_{j-1},x_{j+1},\ldots,x_d)\in (-1,1)^{d-1}$ we have
        \[\int_{-1}^1 f_j(x_1,\ldots,x_d)dx_j=0,\]
        \item For every integer $ m \geq 1 $ we have the estimate $ (\forall\, 1\leq j\leq d)\,\|D^m f_j\|_{\infty}\leq C_{d,m} \|D^m f\|_{\infty}$, where $ C_{d,m} $ depends only on $ d $ and $ m $. 
    \end{enumerate}
\end{lemma}
\begin{proof}
    Fix a smooth function $r \in C_c^\infty ((-1,1))$ such that $\int_{-1}^1 r(t)dt=2$. For each ${1\leq k\leq d}$ define the function
    \[\mathcal{R}_k(x_1,\ldots,x_d):=\prod_{i=1}^k r(x_i),\]
    and moreover put $ \mathcal{R}_0 \equiv 1 $.
    Define $F_0:=f$ and for $1\leq k\leq d$ we define a function
    \[F_k(x_1,\ldots,x_d):=\frac{1}{2^k}\cdot\int_{(-1,1)^k}f(y_1, \ldots, y_k, x_{k+1}, \ldots, x_d) \, dy_1dy_2\cdots dy_k.\]
    Note first that $ F_k(x) $ does not depend on the first $ k $ coordinates $ x_1, \ldots, x_k $, and $ \mathcal{R}_k(x) $ does not depends on $ x_{k+1}, \ldots, x_d $.
    Also note that $\mathcal{R}_k  F_k\in C^{\infty}_{\mathrm{c}}((-1,1)^{d})$ and for every $ x_1,\ldots,x_{k-1},x_{k+1},\ldots,x_d \in (-1,1) $ we have
    \begin{equation*}
    \begin{gathered}
    \int_{-1}^1 \mathcal{R}_k (x_1,\ldots,x_d) F_k (x_1,\ldots,x_d) dx_k \\ 
    =F_k (x_1,\ldots,x_{k-1},0,x_{k+1},\ldots,x_d)  \int_{-1}^1 \mathcal{R}_k (x_1,\ldots,x_d) \,dx_k \\
    = 2F_{k}(x_1,\ldots,x_{k-1},0,x_{k+1},\ldots,x_d) \mathcal{R}_{k-1}(x_1,\ldots,x_{k-1},0,x_{k+1},\ldots,x_d)
    \end{gathered}
    \end{equation*}
    and 
    \begin{equation*}
    \begin{gathered}
    \int_{-1}^1 \mathcal{R}_{k-1} (x_1,\ldots,x_d) F_{k-1} (x_1,\ldots,x_d) \,dx_{k} \\
    =\mathcal{R}_{k-1}(x_1,\ldots,x_{k-1},0,x_{k+1},\ldots,x_d) \int_{-1}^1 F_{k-1}(x_1,\ldots,x_d) dx_{k} \\
    = 2 \mathcal{R}_{k-1}(x_1,\ldots,x_{k-1},0,x_{k+1},\ldots,x_d) F_{k}(x_1,\ldots,x_{k-1},0,x_{k+1},\ldots,x_d).
    \end{gathered}
    \end{equation*}
    Finally, for every $1\leq k\leq d$ we define $f_k:=\mathcal{R}_{k-1}\cdot F_{k-1}-\mathcal{R}_k\cdot F_{k}$.
    From the above equalities we see that $\int_{-1}^1 f_k dx_k=2\mathcal{R}_{k-1} F_{k}-2F_k\mathcal{R}_{k-1}=0$. Moreover,
    \[\sum_{k=1}^d f_k=(f-\mathcal{R}_1F_1)+(\mathcal{R}_1F_1-\mathcal{R}_2F_2)+\ldots+(\mathcal{R}_{d-1}F_{d-1}-\mathcal{R}_dF_d)=f-\mathcal{R}_dF_d=f.\]
    It remains for a given integer $ m \geq 1 $ to bound the $L_{\infty}$-norms of $D^m f_1,\ldots,D^mf_d$. 
    For every $ \ell \geq 0 $ and $ 0 \leq k \leq d $ we have
    $$ \| D^\ell F_k \|_\infty \leq \|D^\ell f\|_\infty ,$$
    and since $ f $ is supported in $ (-1,1)^d $, for every $ \ell \geq 0 $ we have 
    $$ \| D^\ell f \|_\infty \leq \|D^{\ell+1} f\|_\infty .$$
    Hence
    \begin{equation*}
    \begin{gathered}
    \|D^mf_k\|_{\infty} = \| D^m (\mathcal{R}_{k-1}\cdot F_{k-1}-\mathcal{R}_k\cdot F_{k}) \|_\infty \leq C \sum_{\ell=0}^m (\| D^\ell F_{k-1} \|_\infty + \| D^\ell F_{k} \|_\infty) \\
    \leq 2 C \sum_{\ell=0}^m \| D^\ell f \|_\infty \leq 2(m+1)C \| D^m f \|_\infty,
    \end{gathered}
    \end{equation*}
    where $ C $ depends only on the choice of the function $ r $, on $ d $ and $ m $.
\end{proof}

\begin{lemma}\label{EmbeddingLemma}
    There exists a smooth embedding $\Psi:(\mathbb{R}/2\pi\mathbb{Z})\times[-1,1]^{2n-1}\hookrightarrow (-2,2)^{2n}$ and a Hamiltonian diffeomorphism $\Phi\in\Ham_{\mathrm{c}}((-2,2)^{2n})$ such that:
    \begin{enumerate}
        \item $\Phi\circ\Psi(t,x)=\Psi(t+2\sqrt{2}\pi,x)$ for every $ (t,x) \in (\mathbb{R}/2\pi\mathbb{Z})\times[-1,1]^{2n-1} $.
        \item $ \Psi(t,x) = (t,x) $ for every $ (t,x) \in [-1,1]^{2n}\subset (\mathbb{R}/2\pi\mathbb{Z})\times[-1,1]^{2n-1} $.
    \end{enumerate}
\end{lemma}

\begin{proof}
    Consider the annulus $ (\mathbb{R}/2\pi\mathbb{Z})\times [-1,1] \ni (t,x) $ with the symplectic form $ dt \wedge dx $. Since the area of the annulus is 
    $ 4 \pi < 16 $, we can find a symplectic embedding $\psi:(\mathbb{R}/2\pi\mathbb{Z})\times [-1,1]\hookrightarrow (-2,2)^2$. Moreover we may 
    assume that for any $ (t,x) \in [-1,1]^2 \subset (\mathbb{R}/2\pi\mathbb{Z})\times [-1,1] $ we have $ \psi(t,x) = (t,x) $ (see figure 1 below).
    
    \begin{figure}[h]
        \centering
        \includegraphics[scale=1]{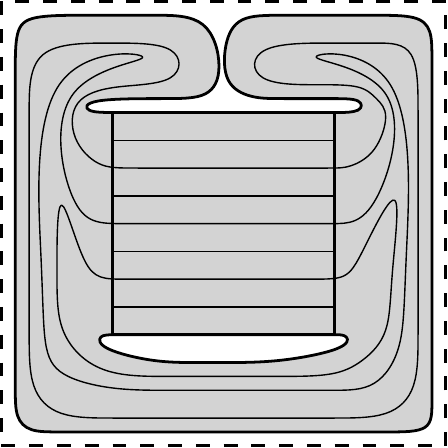}
        \caption{The image of the map $\psi$}
        \label{Sketch}
    \end{figure}
    \noindent Now define a symplectic embedding $\Psi:(\mathbb{R}/2\pi\mathbb{Z})\times[-1,1]^{2n-1}\hookrightarrow (-2,2)^{2n}$ by 
    \[\Psi(t,x_2,\ldots,x_{2n})=\psi(t,x_2) \times\mathrm{Id}(x_3,\ldots,x_{2n}).\]
    Let $H\in C^{\infty}_c((-2,2)^{2n})$ be such that $H\circ\Psi(t,x_2,x_3,\ldots,x_{2n})=x_2 $. Finally, define $\Phi$ as a time-$2\sqrt{2}\pi$ map 
    of the Hamiltonian flow generated by $H$. Both stated properties are clearly satisfied.
\end{proof}
\begin{lemma}\label{rotation-Lemma}
    For every compactly supported smooth function $f:(\mathbb{R}/2\pi\mathbb{Z})\times (-1,1)^m \rightarrow\mathbb{R}$ which satisfies $\int_{0}^{2\pi}f(t,x)dt=0$ for all $x\in (-1,1)^m$, there exists a compactly supported smooth function $g:(\mathbb{R}/2\pi\mathbb{Z})\times (-1,1)^m \rightarrow\mathbb{R}$ such that
    \begin{enumerate}
        \item $f(t,x)=g(t+2\sqrt{2}\pi,x)-g(t,x)$ for every $ (t,x) \in (\mathbb{R}/2\pi\mathbb{Z})\times (-1,1)^m $.
        \item $\|Dg\|_{\infty}\leq 2 \|D^3 f\|_{\infty}$.
    \end{enumerate}
\end{lemma}
\begin{proof}
    Denote $\alpha=2\sqrt{2}\pi$. Let
    \[\widehat{f}(n,x):=\frac{1}{2\pi}\int_{0}^{2\pi}f(t,x)e^{-int}dt \]
    be the $n$-th Fourier coefficient of $f$ (with respect to the variable $ t $). The condition $\int_{0}^{2\pi}f(t,x)dt=0$ ($\forall x\in (-1,1)^m$) means that 
    $ \widehat{f}(0,x) = 0 $ for all $ x\in (-1,1)^m $. Now we define
    \begin{equation} \label{eq:deffh}
     g(t,x):=\sum_{n \neq 0}\frac{\widehat{f}(n,x)}{e^{in\alpha}-1}\cdot e^{int}.
    \end{equation}
    For addressing the convergence and properties of this series, we first derive (well-known) lower bounds for $ |e^{in\alpha}-1| $.
    
    Let $ n \in \mathbb{Z} \setminus \{ 0 \} $. Denote by $ \lambda \in (0,1/2) $ the distance from $ \sqrt{2}n $ to $ \mathbb{Z} $. Then 
    $$ | e^{in\alpha} -1| = 2|\sin(n\alpha /2)| = 2|\sin(\sqrt{2} \pi n)| = 2 |\sin(\pi \lambda)| \geqslant 4 \lambda .$$
    If $ m \in \mathbb{Z} $ is such that $ \lambda = |\sqrt{2}n-m| $ then $ m $ has the same sign as $ n $, and $ |m| < \sqrt{2}|n|+1 $. Therefore
    $$ \lambda = |\sqrt{2}n-m| = \frac{|2n^2-m^2|}{|\sqrt{2}n+m|} \geqslant \frac{1}{|\sqrt{2}n+m|} > \frac{1}{2\sqrt{2}|n|+1} .$$
    We conclude 
    $$ | e^{in\alpha} -1| > \frac{4}{2\sqrt{2}|n|+1} ,$$
    hence the smoothness of the function $ f $ implies smoothness of $ g $ and the uniform convergence in $(\ref{eq:deffh})$. In particular,
    $$ g(t+\alpha,x)-g(t,x) = \sum_{n \neq 0} \widehat{f}(n,x) e^{int} = f(t,x) .$$ Moreover, it is clear from the definition of $ g $ that it has a compact support in 
    $ {(\mathbb{R}/2\pi\mathbb{Z})\times (-1,1)^m} $.
    
    It remains to show $\|Dg\|_{\infty}\leq 2 \|D^3 f\|_{\infty}$. First of all, we have 
    \[ \frac{\partial g}{\partial t}(t,x) = i \sum_{n \neq 0}\frac{n \widehat{f}(n,x)}{e^{in\alpha}-1} e^{int} = 
    -\sum_{n \neq 0} \frac{\widehat{\partial^3f/\partial t^3}(n,x)}{n^2(e^{in\alpha}-1)} e^{int} , \]
    hence we get
    \[ \left| \frac{\partial g}{\partial t}(t,x) \right| \leq \sum_{n \neq 0}\frac{|\widehat{\partial^3f/\partial t^3}(n,x)|}{n^2 |e^{in\alpha}-1|} \leq
    \left( \sum_{n \neq 0} |\widehat{\partial^3f/\partial t^3}(n,x)|^2 \right)^{\frac{1}{2}}  \left( \sum_{n \neq 0}\frac{1}{n^4 |e^{in\alpha}-1|^2} \right)^{\frac{1}{2}} \]
    \[ =  \left( \frac{1}{2\pi} \int_0^{2\pi} \left|\frac{\partial^3f}{\partial t^3}(s,x) \right|^2 \, ds \right)^{\frac{1}{2}}  
    \left( \sum_{n \neq 0}\frac{1}{n^4 |e^{in\alpha}-1|^2} \right)^{\frac{1}{2}} 
    \leq \|D^3 f\|_{\infty} \left( \sum_{n \neq 0}\frac{(2\sqrt{2}|n|+1)^2}{16 n^4} \right)^{\frac{1}{2}} \]
    \[ \leq \frac{2\sqrt{2}+1}{4}  \|D^3 f\|_{\infty} \left( \sum_{n \neq 0} \frac{1}{n^2} \right)^{\frac{1}{2}} 
    =\frac{(2\sqrt{2}+1) \pi }{4\sqrt{3}}  \|D^3 f\|_{\infty} \leq 2 \|D^3 f\|_{\infty} .\]
    In addition, for every $1\leq k \leq m$ we have
    \[ \frac{\partial g}{\partial x_k}(t,x) = \sum_{n \neq 0} \frac{\widehat{\partial f/\partial x_k}(n,x)}{e^{in\alpha}-1} e^{int} 
    = -\sum_{n \neq 0} \frac{\widehat{\partial^3f/\partial x_k \partial t^2}(n,x)}{n^2(e^{in\alpha}-1)} e^{int} , \]
    and hence in a similar way we conclude
    $ \left| \frac{\partial g}{\partial x_k}(t,x) \right| \leq 2 \|D^3 f\|_{\infty}$.
\end{proof}

\begin{lemma}\label{embeddingLemma}
    Let $h : \mathbb{R}^{2n} \rightarrow \mathbb{R} $ be a smooth function supported in $ (-1,1)^{2n} $, such that $|\partial h/\partial x_1|<1/2$. Then there exists a unique diffeomorphism $\Psi_h:\mathbb{R}^{2n} \rightarrow \mathbb{R}^{2n}$ which satisfies
    \begin{equation} \label{eq:Psihdef}
    \begin{gathered}
     \Psi_h\left(\phi^t_H(x_1,1,x_3,x_ 4,\ldots,x_{2n-1},x_{2n})\right) =\phi^t_{x_1}(x_1,1,x_3,x_ 4,\ldots,x_{2n-1},x_{2n}) \\
     =(x_1,1-t,x_3,x_ 4,\ldots,x_{2n-1},x_{2n}) 
    \end{gathered}
    \end{equation}
    for every $ t,x_1,x_3,x_4,\ldots,x_{2n} \in \mathbb{R} $, 
    where $H:=x_1+h$. Moreover, $ \Psi_h $ preserves the standard symplectic form $\omega =\sum_{i=1}^n dx_{2i-1} \wedge dx_{2i} $ and satisfies
    \begin{enumerate}
        \item $\Psi_h^* x_1-x_1=h $,
        \item $(\forall\,1\leq i\leq 2n)\,|\Psi_{h}^*x_{i}-x_{i}|\leq 2|1-x_2|\cdot\|dh\|_{\infty}$.
    \end{enumerate}
\end{lemma}
\begin{proof}
Let $V=\{x_2=1\}\subset\mathbb{R}^{2n}$. Since we have $\partial H/\partial x_1>1/2$, we conclude that the second coordinate of the Hamiltonian vector field $ X_H $ at each point is bounded from above by $ -1/2 $.  Hence the map
\[F:\mathbb{R}\times V\rightarrow\mathbb{R}^{2n},\quad F(t,x)=\phi^t_H(x)\]
is a diffeomorphism, and the diffeomorphism $ \Psi_h $ is uniquely defined by $ (\ref{eq:Psihdef}) $. Let us show that $ \Psi_h $ is symplectic.
    \begin{enumerate}[label=(\Roman*)]
        \item \underline{$\omega(X_1,X_2)=\omega((\Psi_{h})_* X_1,(\Psi_{h})_* X_2)$ for $X_1,X_2\in T\phi^t_H(V)$:}\\
        Note that $\Psi_{h}|_{\phi^t_H(V)}=\phi^t_{x_1}\circ\left(\phi^t_H\right)^{-1}|_{\phi^t_H(V)}$, so the equality follows from the fact that $\phi^t_{x_1}\circ\left(\phi^t_H\right)^{-1}$ is symplectic.
        \item \underline{$\omega(X_H,X)=\omega((\Psi_{h})_* X_H,(\Psi_{h})_* X)$ for $X\in T\phi^t_H(V)$:}\\
        Since $ \phi^t_H $ and $ \phi^t_{x_1} $ preserve $ \omega $, for each $ Y \in TV $ we get  
        $$ \omega(X_H,(\phi^t_H)_* Y)=\omega(X_H,Y)= \omega(X_{x_1},Y)=\omega(X_{x_1},(\phi^t_{x_1})_*Y), $$
        where the second equality holds since $ H = x_1 $ near $ V $.
        On the other hand, by differentiating the equality $(\ref{eq:Psihdef})$ with respect to $ t $ we get $(\Psi_{h})_* X_H= X_{x_1} \circ \Psi_{h} $ on $ \phi^t_H(V)$. Moreover, from $(\ref{eq:Psihdef})$ we get $(\Psi_{h})_*(\phi^t_{H})_* Y=(\phi^t_{x_1})_* Y$ for each $ Y \in TV $. Hence for every $X \in T\phi^t_H(V)$,
        denoting $ Y := (\phi^t_H)^*X \in TV $, we get
        $$ \omega(X_H,X) =  \omega(X_H,(\phi^t_H)_* Y) = \omega(X_{x_1},(\phi^t_{x_1})_*Y) = \omega((\Psi_{h})_* X_H,(\Psi_{h})_* X)  .$$
    \end{enumerate}
    From (I) and (II) we conclude that $\Psi_{h}$ preserves $\omega$.\\
    
    To show the point 1. in the lemma, note that the equality
    \[(\Psi_{h}\circ\phi^t_H)^*x_1=x_1= H=H\circ\phi^t_H\]
    holds on $V$ for every $ t \in \mathbb{R} $. Since the map $F:\mathbb{R}\times V\rightarrow\mathbb{R}^{2n},(t,x)\mapsto\phi^t_H(x)$ is a diffeomorphism, we conclude that $\Psi_{h}^*x_1=H=x_1+h$, or equivalently $\Psi_{h}^*x_1-x_1=h$.\\
    
    Finally, we get the following inequalities on $V$, for every $ t \in \mathbb{R} $. For $ 1 \leq i \leq n $, 
    \[ |(\Psi_{h}\circ\phi^t_H)^*x_{2i-1}-(\phi^t_H)^*x_{2i-1}|=|x_{2i-1}-(\phi^t_H)^*x_{2i-1}|=\Big|\int_0^t\frac{\partial h}{\partial x_{2i}}\circ\phi^s_H \, ds\Big|\leq |t|\cdot\|Dh\|_{\infty}.\]
    We also have
    \[|(\Psi_{h}\circ\phi^t_H)^*x_2-(\phi^t_H)^*x_2|=|1-t-(\phi^t_H)^*x_2|=\Big|-t+\int_0^t(1+\frac{\partial h}{\partial x_1}\circ\phi^s_H) \, ds \Big|\leq |t|\cdot\|Dh\|_{\infty} .\]
    Finally, for all $ 2\leq i\leq n $ we have
    \[ |(\Psi_{h}\circ\phi^t_H)^*x_{2i}-(\phi^t_H)^*x_{2i}|=|x_{2i}-(\phi^t_H)^*x_{2i}|=\Big|\int_0^t\frac{\partial h}{\partial x_{2i-1}}\circ\phi^s_H \, ds \Big|\leq |t|\cdot\|Dh\|_{\infty}.\]
    
    \noindent For $p=(x_1,x_2,\ldots,x_{2n-1},x_{2n})\in\mathbb{R}^{2n}$ define $t:=\pi_1\circ F^{-1}(p)$, where $\pi_1$ is the projection to the first coordinate in $\mathbb{R}\times V$. Since $\partial H/\partial x_1>1/2$ we conclude that $ |t| \leqslant 2 |x_2-1|$, therefore we get the point 2. in the lemma.
\end{proof}

\subsection{Proof of Proposition \ref{MainProp1}}

We first claim that by rescaling we can without loss of generality assume that $ L = 1 $. Indeed, let $L>0$ and $f\in C^{\infty}_{\mathrm{c}}((-L,L)^{2n})$. Define $\widehat{f}\in C^{\infty}_{\mathrm{c}}((-1,1)^{2n})$, $\widehat{f}(x):=\frac{1}{L}f(Lx)$. Applying the case $ L = 1 $ of the statement of the proposition to $\widehat{f}$, we obtain $\widehat{g}_1,\ldots,\widehat{g}_{2n}\in C^{\infty}_{\mathrm{c}}((-2,2)^{2n})$ and $\widehat{\Phi}_1,\ldots,\widehat{\Phi}_{2n}\in\Ham_{\mathrm{c}}((-2,2)^{2n})$ such that $f=\sum_{i=1}^{2n}\widehat{\Phi}^*_i\,\widehat{g}_i-\widehat{g}_i$ and $\|D\widehat{g}_i\|_{\infty}\leq C\|D^3\widehat{f}_i\|_{\infty}$. For each $1\leq i\leq 2n$ define $\Phi_i\in\Ham_{\mathrm{c}}((-2L,2L)^{2n})$, $\Phi_i(x):=L\widehat{\Phi}_i(x/L)$, and $g_i\in C^{\infty}_{\mathrm{c}}((-2L,2L)^{2n})$, $g_i(x):=L\widehat{g}_i(x/L)$. It remains to notice that $\|Dg_i\|_{\infty}=\|D\widehat{g}_i\|_{\infty}$ and $L^2\|D^3 f\|_\infty=\|D^3\widehat{f}\|_\infty$.

Now we concentrate on the case of $ L = 1 $, and let $f: (-2,2)^{2n} \rightarrow \mathbb{R} $ be a smooth function supported in $ (-1,1)^{2n} $ such that $\int_{(-1,1)^{2n}}f\omega^{n}=0$. Apply Lemma \ref{PoincareLemma} to get a collection $f_1,\ldots,f_{2n}\in C^{\infty}_{\mathrm{c}}((-1,1)^{2n})$ such that $f=\sum_{j=1}^{2n} f_j$ and
\begin{itemize}
    \item $(\forall\, 1\leq j\leq 2n)$ and any $(x_1,\ldots,x_{j-1},x_{j+1},\ldots,x_{2n})\in (-1,1)^{2n-1}$ we have
    \[\int_{-1}^1 f_j(x_1,\ldots,x_{2n})dx_j=0,\]
    \item $(\forall\, 1\leq j\leq 2n)\,\|D^3 f_j\|_{\infty}\leq C \|D^3 f\|_{\infty}$.
\end{itemize}
Using Lemma \ref{EmbeddingLemma} and Lemma \ref{rotation-Lemma} we are going to construct $\Phi_1,\ldots,\Phi_{2n}\in \Ham_{\mathrm{c}}((-2,2)^{2n})$ and $g_1,\ldots,g_{2n}\in C^{\infty}_{\mathrm{c}}((-2,2)^{2n})$ such that $f_i=\Phi_i^* g_i-g_i$. We only show the case $i=1$. 

Let $\Psi: (\mathbb{R}/2\pi\mathbb{Z})\times[-1,1]^{2n-1}\hookrightarrow (-2,2)^{2n}$ be a smooth embedding given by Lemma \ref{EmbeddingLemma}, and let $\Phi\in\Ham_{\mathrm{c}}((-2,2)^{2n})$ be a Hamiltonian diffeomorphism given by the lemma. Apply Lemma \ref{rotation-Lemma} to the function \[\widetilde{f}_1: (\mathbb{R}/2\pi\mathbb{Z})\times (-1,1)^{2n-1}\rightarrow\mathbb{R},\quad \widetilde{f_1}:=f_1\circ\Psi,\]
to get a compactly supported function $\widetilde{g}_1 : (\mathbb{R}/2\pi\mathbb{Z})\times(-1,1)^{2n-1} \rightarrow \mathbb{R} $ which satisfies
\begin{itemize}
    \item $\widetilde{f}_1(t,x)=\widetilde{g}_1(t+2\sqrt{2}\pi,x)-\widetilde{g}_1(t,x)$,
    \item $\|D\widetilde{g}_1\|_{\infty}\leq 2 \|D^3\widetilde{f}_1\|_{\infty}$.
\end{itemize}

\noindent Finally, denoting $ W := \Psi ((\mathbb{R}/2\pi\mathbb{Z})\times(-1,1)^{2n-1}) $, define $g_1\in C^{\infty}_{\mathrm{c}}(W)\subset C^{\infty}_{\mathrm{c}}((-2,2)^{2n})$ via the equality $g_1\circ\Psi=\widetilde{g}_1$. Therefore we have 
$$ f_1\circ\Psi (t,x) = \widetilde{f}_1(t,x) = \widetilde{g}_1(t+2\sqrt{2}\pi,x)-\widetilde{g}_1(t,x) = g_1\circ\Phi \circ\Psi (t,x) - g_1\circ\Psi (t,x).$$
Since $f_1$ and $g_1$ are both supported inside $W$ we conclude $f_1=\Phi^*g_1-g_1$. Moreover, 
\begin{equation*}
\begin{gathered}
\|Dg_1\|_{\infty}\leq \widetilde{C}_1 \|D\widetilde{g}_1\|_{\infty} \leq 2 \widetilde{C}_1 \|D^3 \widetilde{f}_1\|_{\infty} \\
\leq \widetilde{C}_2 (\|D f_1\|_{\infty} + \|D^2 f_1\|_{\infty} +\|D^3 f_1\|_{\infty}) \leq 3\widetilde{C}_2 \|D^3 f_1\|_{\infty} \leq  3 C \widetilde{C}_2 \|D^3f \|_\infty
\end{gathered}
\end{equation*} 
where the constants $ \widetilde{C}_1, \widetilde{C}_2 $ depend only on the map $ \Psi $.

\subsection{Proof of Proposition \ref{MainProp2}}

We first claim that by rescaling we can without loss of generality assume that $ L = \lambda = 1 $. Indeed, let $L>0$ and $\lambda\neq 0$, and assume that $h\in C^{\infty}_{\mathrm{c}}((-L,L)^{2n})$ satisfies $\|Dh\|_{\infty}\leq\frac{|\lambda|}{8}$. Define $\widetilde{h}\in C^{\infty}_{\mathrm{c}}((-1,1)^{2n}) $ by $ \widetilde{h}(x) = \frac{1}{L}h(Lx)$, and note that $\|D\widetilde{h}\|_{\infty}=\|Dh\|_{\infty}$. Apply the case $ L = \lambda = 1 $ of the proposition to the function $\frac{1}{\lambda}\cdot\widetilde{h}$ to get $\widetilde{\Phi}\in\Ham_{\mathrm{c}}((-2,2)^{2n})$ such that the equality $\widetilde{\Phi}^*\,x_1-x_1=\frac{1}{\lambda}\cdot\widetilde{h}$ holds on $(-1,1)^{2n}$. Equivalently, the equality $\widetilde{\Phi}^*\,(\lambda x_1)-(\lambda x_1)=\widetilde{h}$ holds on $(-1,1)^{2n}$. It remains to define $\Phi\in\Ham_{\mathrm{c}}((-2L,2L)^{2n})$, $\Phi(x):=L\widetilde{\Phi}(x/L)$.

Now we concentrate on the case of $ L = \lambda = 1 $, and let $h\in C^{\infty}_{\mathrm{c}}((-1,1)^{2n})$ such that $\|Dh\|_{\infty}\leq\frac{1}{8}$. 
For each $t\in[0,1]$ apply Lemma \ref{embeddingLemma} to the function $th$ to get a symplectomorphism $ \Psi_{th} $ of $ \mathbb{R}^{2n} $ satisfying:
\begin{itemize}
    \item $\Psi_{th}^*x_1-x_1=th$,
    \item $(\forall\,1\leq i\leq 2n)\,|\Psi_{th}^*x_i-x_i|\leq 2|1-x_2|\cdot t\|Dh\|_{\infty}$.
\end{itemize}
It is easy to see that the proof of Lemma \ref{embeddingLemma} guarantees that $t\mapsto\Psi_{th}$ ($t \in [0,1]$) is a smooth (possibly not compactly supported) Hamiltonian flow on $ \mathbb{R}^{2n} $. Let $H:\mathbb{R}^{2n} \times [0,1] \rightarrow\mathbb{R}$ be the (time dependent) Hamiltonian function generating the flow $t\mapsto\Psi_{th}$. Let $\chi\in C^{\infty}_{\mathrm{c}}((-2,2)^{2n})$ be the function which equals $1$ on $(-3/2,3/2)^{2n}$, and consider another Hamiltonian flow $t\mapsto\phi^t_{\chi H}$. Note that for any $x=(x_1,x_2,\ldots,x_{2n})\in (-1,1)^{2n}$ and $ t \in [0,1] $ we have
\[|\Psi_{th}^*x_i-x_i|\leq 2|1-x_2|\cdot\|Dh\|_{\infty}\leq 4\|Dh\|_{\infty}\leq\frac{1}{2},\]
which implies
\[ \quad\Psi_{th}((-1,1)^{2n})\subset (-3/2,3/2)^{2n}.\]
In particular $\Psi_{th}|_{(-1,1)^{2n}}=\phi^t_{\chi H}|_{(-1,1)^{2n}}$, therefore we can define $\Phi:=\phi^1_{\chi H}$.

\subsection{Proof of Theorem \ref{ThmMicro}}

We define $ C:= \frac{1}{8C_1} $ where the constant $ C_1 $ is given by Proposition \ref{MainProp1}. Now let $f : (-4\varepsilon,4\varepsilon)^{2n} \rightarrow \mathbb{R} $ 
be a smooth function supported in $ (-\varepsilon,\varepsilon)^{2n} $ such that $\|D^3f\|_{\infty}\leq\frac{C}{\varepsilon^3}$ and $\int_{(-\varepsilon,\varepsilon)^{2n}}f\omega^n=0$. Applying Proposition \ref{MainProp1} to $f$ we get a collection of smooth functions $g_1,\ldots,g_{2n}\in C^{\infty}_{\mathrm{c}}((-2\varepsilon,2\varepsilon)^{2n})$ such that $f=\sum_{i=1}^{2n}\Phi_i^* g_i-g_i$, for some Hamiltonian diffeomorphisms $\Phi_1,\ldots,\Phi_{2n}\in\Ham_{\mathrm{c}}((-2\varepsilon,2\varepsilon)^{2n})$, and such that moreover for each $1\leq i\leq 2n$ we have $\|Dg_i\|_{\infty}\leq  C_1\varepsilon^2\cdot\|D^3f\|_{\infty}$ and hence 
\[\|Dg_i\|_{\infty}\leq C_1\varepsilon^2\cdot\|D^3f\|_{\infty}\leq\frac{1}{8\varepsilon}.\]

This allows us to apply Proposition \ref{MainProp2} to the function $g_i$, $L=2\varepsilon$ and $\lambda=\frac{1}{\varepsilon}$. We get a Hamiltonian diffeomorphism $\Psi_i\in\Ham_{\mathrm{c}}((-4\varepsilon,4\varepsilon)^{2n})$ which satisfies $\Psi_i^*v-v=g_i$ on $(-2\varepsilon,2\varepsilon)^{2n}$. Since the support of $\Phi_i$ is inside $(-2\varepsilon,2\varepsilon)^{2n}$, the equality
\[\Phi_i^*g_i-g_i=\Phi_i^*(\Psi_i^*v-v)-(\Psi_i^*v-v)\]
holds everywhere. Finally, we get the required representation of $f$ by defining $\Phi_{2i-1,+}:=\Psi_i\circ\Phi_i$, $\Phi_{2i,+}:=\mathrm{Id}$, $\Phi_{2i-1,-}:=\Phi_i$ and $\Phi_{2i,-}:=\Psi_i$ for $1\leq i\leq 2n$.

\section{A local version of the main result} \label{Section:Local}

\begin{theorem}[Local statement]\label{ThmLocal}
    Let $ L > 0 $. There exists $N=N(n)\in\mathbb{N}$ such that for any smooth function $f : (-8L,8L)^{2n} \rightarrow \mathbb{R} $ supported in $ (-L,L)^{2n} $ and 
    which satisfies $\int_{(-L,L)^{2n}}f\omega^n=0$ and $\|f\|_{\infty}\leq L$, one can write
    \begin{equation} \label{eq:thmlocal}
     f=\sum_{i=1}^N\Phi_{i,+}^*x_1-\sum_{i=1}^N\Phi_{i,-}^*x_1,
    \end{equation}
    for some Hamiltonian diffeomorphisms $\Phi_{i,\pm}\in\Ham_{\mathrm{c}}((-8L,8L)^{2n})$.
\end{theorem}

\subsection{Proof of Theorem \ref{ThmLocal}}

Without loss of generality we may assume that $ L = 1 $ and that $ f $ is a non-zero function.
Consider a smooth function $r:\mathbb{R} \rightarrow [0,1]$ such that
\[r(t)=
        \begin{cases}
            1,& \text{for } t\in[-1/3,1/3],\\
            0,& \text{for } t\in (-\infty,-2/3]\cup[2/3,\infty),
        \end{cases}\]
and such that $\sum_{i\in\mathbb{Z}}r(t+i)=1$. Note that this in particular means that $ \int_{-1}^1 r(t) \, dt = 1 $. Now we define 
$$ \mathcal{R}(x)=\mathcal{R}(x_1,\ldots,x_{2n})=\prod_{i=1}^{2n}r(x_i) ,$$
and for a sufficiently small $\varepsilon>0$ we put
\[\mathcal{R}^{\varepsilon}(x)=\mathcal{R}(x/\varepsilon) = \prod_{i=1}^{2n}r(x_i/\varepsilon).\]
Note that we have $\sum_{v\in\varepsilon\mathbb{Z}^{2n}}\mathcal{R}^{\varepsilon}(x-v)=1$ and 
\begin{equation} \label{eq:intR}
\int_{\mathbb{R}^{2n}}\mathcal{R}^{\varepsilon}\omega^{n}=\varepsilon^{2n}.
\end{equation}

\noindent For each $\lambda\in\mathcal{X}=\{0,1\}^{2n}$ consider the finite grid $\Gamma^{\varepsilon}_{\lambda}\subset (-1,1)^{2n}$ defined as:
\[\Gamma^{\varepsilon}_{\lambda}:=\left(\varepsilon\cdot\lambda+2\varepsilon\mathbb{Z}^{2n}\right)\cap (-1,1)^{2n}.\]

\begin{lemma}\label{Lemma1.12}
    Given $f$ as in Theorem \ref{ThmLocal}, there is an $\varepsilon>0$ with the following property. For each $v\in\bigcup_{\lambda\in\mathcal{X}}\Gamma^{\varepsilon}_{\lambda}$ there exists a smooth function $g_v$ supported in $v+[-2\varepsilon/3,2\varepsilon/3]^{2n}$, such that:
    \begin{itemize}
        \item $f=\sum_{\lambda\in\mathcal{X}}f_{\lambda}$, where $f_{\lambda}:=\sum_{v\in\Gamma^{\varepsilon}_{\lambda}}g_v$,
        \item $(\forall\,\lambda\in\mathcal{X})$ $\int_{\mathbb{R}^{2n}}f_{\lambda}\omega^{n}=0$,
        \item $(\exists\, C=C(n))(\forall\,v\in\bigcup_{\lambda\in\mathcal{X}}\Gamma^{\varepsilon}_{\lambda})\,\|D^3f_\lambda\|_{\infty}\leq\frac{C\|f\|_{\infty}}{\varepsilon^3}$, and $\|g_v\|_{\infty}\leq C\|f\|_{\infty}$.
    \end{itemize}
\end{lemma}

\begin{proof}
    
For each $v\in\bigcup_{\lambda\in\mathcal{X}}\Gamma^{\varepsilon}_{\lambda}$ define $G_v:=f(x)\cdot\mathcal{R}^{\varepsilon}(x-v)$ and let
\[F_{\lambda}=\sum_{v\in\Gamma^{\varepsilon}_{\lambda}}G_v.\]
Note that $\sum_{\lambda\in\mathcal{X}}F_{\lambda}=f$, and $\mathrm{supp}(G_v)\subset v+[-2\varepsilon/3,2\varepsilon/3]^{2n}$ for every $v\in\bigcup_{\lambda\in\mathcal{X}}\Gamma^{\varepsilon}_{\lambda}$. Moreover, we have
\begin{equation}\label{eqD^4G_v}
    \|D^3G_v\|_{\infty} \leq \frac{1}{\varepsilon^3}\|f\|_{\infty}\cdot\|D^3\mathcal{R}\|_{\infty}+\mathcal{O}\Big(\frac{1}{\varepsilon^2}\Big)\,\text{ and }\|G_v\|_{\infty}\leq \|f\|_{\infty}.
\end{equation}

Now let $ \lambda \in \mathcal{X} $. For $ v \in \Gamma^{\varepsilon}_{\lambda} $, denoting $ Q_v := v + [-\varepsilon,\varepsilon]^{2n} $, by $(\ref{eq:intR})$ we get
$$ \int_{Q_v} (\mathcal{R}^{\varepsilon}(x-v)-2^{-2n})\omega^n = 0 ,$$
and hence
$$ \int_{Q_v} f(x) \mathcal{R}^{\varepsilon}(x-v)\omega^n -  2^{-2n} \int_{Q_v} f(x) \omega^n =   \int_{Q_v} (f(x)-f(v)) (\mathcal{R}^{\varepsilon}(x-v) - 2^{-2n}) \omega^n  = \mathcal{O}(\varepsilon^{2n+1}) .$$
Therefore we have 
$$ \int_{\mathbb{R}^{2n}} F_\lambda \omega^n = \sum_{v \in \bigcup_{\lambda\in\mathcal{X}}\Gamma^{\varepsilon}_{\lambda}} \int_{Q_v} f(x) \mathcal{R}^{\varepsilon}(x-v)\omega^n = 2^{-2n} \left( \sum_{v \in \bigcup_{\lambda\in\mathcal{X}}\Gamma^{\varepsilon}_{\lambda}} \int_{Q_v} f(x) \omega^n \right) + \mathcal{O}(\varepsilon) = \mathcal{O}(\varepsilon), $$
or in short, we get 
\begin{equation} \label{eq:intF-lam}
\int_{\mathbb{R}^{2n}} F_\lambda \omega^n = \mathcal{O}(\varepsilon) 
\end{equation}
for each $ \lambda \in \mathcal{X} $.

Consider the finite graph $(\mathcal{V},\mathcal{E})$ defined as
\[\mathcal{V}=\bigcup_{\lambda\in\mathcal{X}}\Gamma^{\varepsilon}_{\lambda},\quad\mathcal{E}=\{\{v,w\}\mid v\neq w\text{ and }v-w\in[-\varepsilon,\varepsilon]^{2n}\}.\]
For each edge $\{u,v\}\in\mathcal{E}$ define the function
\[G_{u,v}(x):=\frac{1}{(\varepsilon/6)^{2n}}\mathcal{R}^{\frac{\varepsilon}{6}}\Big(x-\frac{u+v}{2}\Big),\]
and note that this function satisfies the following:
\begin{enumerate}\label{eqD^4G_u,v}
    \item $\int_{\mathbb{R}^{2n}}G_{u,v}\omega^n=1$,
    \item $\mathrm{supp}(G_{u,v})\subset (u+[-2\varepsilon/3,2\varepsilon/3]^{2n}) \cap (v+[-2\varepsilon/3,2\varepsilon/3]^{2n})$,
    \item $\|D^3G_{u,v}\|_{\infty} \leq \frac{\|D^3\mathcal{R}\|_{\infty}}{(\varepsilon/6)^{2n+3}}+\mathcal{O}\Big(\frac{1}{\varepsilon^{2n+2}}\Big)\,\text{ and }\|G_{u,v}\|_{\infty}=(6/\varepsilon)^{2n}$.
\end{enumerate}
Let $N_{\mu,\lambda}$ be the number of edges between vertices in $\Gamma_{\mu}$ and $\Gamma_{\lambda}$. Moreover, define $I_{\lambda}:=\int_{\mathbb{R}^{2n}} F_{\lambda}\omega^n$ and note that we have $\sum_{\lambda\in\mathcal{X}}I_{\lambda}=0$, and by $(\ref{eq:intF-lam})$ we have $ I_\lambda = \mathcal{O}(\varepsilon) $. Finally, we define
\[(\forall \lambda \in \mathcal{X} \; \forall u\in \Gamma^{\varepsilon}_{\lambda})\quad g_u:=G_u+\sum_{\mu\in\mathcal{X}}\sum_{\substack{\{u,v\}\in\mathcal{E}\\v\in\Gamma_{\mu}}}\frac{I_{\mu}-I_{\lambda}}{2^n N_{\mu,\lambda}}\cdot G_{u,v},\]
and then for every $ \lambda \in \mathcal{X} $ we put $f_{\lambda}=\sum_{v\in\Gamma^{\varepsilon}_{\lambda}}g_v$. By mutual cancellation of the terms $ \frac{I_{\mu}-I_{\lambda}}{2^n N_{\mu,\lambda}}\cdot G_{u,v} $ we obtain $\sum_{\lambda\in\mathcal{X}}f_{\lambda}=\sum_{\lambda\in\mathcal{X}}F_{\lambda}=f$, and moreover we get 
\[(\forall\lambda\in\mathcal{X})\quad\int_{\mathbb{R}^{2n}}f_{\lambda}\omega^n=I_{\lambda}+\sum_{\substack{\mu\in\mathcal{X}\\\mu\neq\lambda}}\sum_{\substack{\{u,v\}\in\mathcal{E}\\u\in\Gamma^{\varepsilon}_{\lambda},v\in\Gamma^{\varepsilon}_{\mu}}}\frac{I_{\mu}-I_{\lambda}}{2^n N_{\mu,\lambda}}=\]
\[=I_{\lambda}+\sum_{\substack{\mu\in\mathcal{X}\\\mu\neq\lambda}}\frac{I_{\mu}-I_{\lambda}}{2^n}=\frac{1}{2^n}I_{\lambda}+\frac{1}{2^n}\sum_{\substack{\mu\in\mathcal{X}\\\mu\neq\lambda}}I_{\mu}=0.\]
Also note that $N_{\mu,\lambda}>(1/\varepsilon)^{2n}$ and $ I_{\lambda} = \mathcal{O}(\varepsilon) $. Combining this with (\ref{eqD^4G_v}) and the properties of $G_{u,v}$ listed above, we get that for a small enough $\varepsilon$, 
\[\|g_u\|_{\infty}\leq 2 \|f\|_{\infty},\]
\[\|D^3 g_u\|_{\infty}\leq (1+\|D^3\mathcal{R}\|_{\infty}) \frac{\|f\|_{\infty}}{\varepsilon^3}.\]
Note that $f_{\lambda}=\sum_{v\in\Gamma^{\varepsilon}_{\lambda}}g_v$, and $\mathrm{supp}\,f_{\lambda}=\bigsqcup_{v\in\Gamma^{\varepsilon}_{\lambda}}\mathrm{supp}\,g_u$, therefore it is enough to take $C=2 + \|D^3\mathcal{R}\|_{\infty}$.

\end{proof}

\begin{lemma}\label{LemmaPartitionOfFunction}
    For each $\lambda\in\mathcal{X}$ there exist:
    \begin{enumerate}
        \item a Hamiltonian diffeomorphism $\Phi_{\lambda}\in\Ham_{\mathrm{c}}((-2,2)^{2n})$,
        \item a finite subset $\widetilde{\Gamma}^{\varepsilon}_{\lambda}\subset(\varepsilon\mathbb{Z}^{2n}\setminus\Gamma^{\varepsilon}_{\lambda})\cap (-1,1)^{2n} $,
        \item a collection of smooth functions $\{f_{\lambda,v}\mid v\in\Gamma^{\varepsilon}_{\lambda}\cup\widetilde{\Gamma}^{\varepsilon}_{\lambda}\}$,
    \end{enumerate}
    such that the following holds:
    \begin{enumerate}
        \item $\Phi_{\lambda}^*f_{\lambda}=\sum_{v\in\Gamma^{\varepsilon}_{\lambda}\cup\widetilde{\Gamma}^{\varepsilon}_{\lambda}} f_{\lambda,v}$,
        \item $(\forall\,v\in\Gamma^{\varepsilon}_{\lambda}\cup\widetilde{\Gamma}^{\varepsilon}_{\lambda})\,\,\, \mathrm{supp}(f_{\lambda,v})\subset v+[-2\varepsilon/3,2\varepsilon/3]^{2n}$,
        \item $(\forall\,v\in\Gamma^{\varepsilon}_{\lambda}\cup\widetilde{\Gamma}^{\varepsilon}_{\lambda})\, \int_{\mathbb{R}^{2n}}f_{\lambda,v} \omega^{n}=0$,
        \item $(\exists\, C=C(n))(\forall\,v\in\Gamma^{\varepsilon}_{\lambda}\cup\widetilde{\Gamma}^{\varepsilon}_{\lambda})\,\|D^3f_{\lambda,v}\|_{\infty}\leq\frac{C\|f\|_{\infty}}{\varepsilon^3}$.
    \end{enumerate}
\end{lemma}

\begin{proof}
    Arrange the elements of $\Gamma^{\varepsilon}_{\lambda}$ in a sequence $v_1,v_2,\ldots,v_{N}$, such that the cubes $v_i+[-\varepsilon,\varepsilon]^{2n}$ and $v_{i+1}+[-\varepsilon,\varepsilon]^{2n}$ share a common $(2n-1)$-dimensional face. Apply Lemma \ref{Lemma1.12} and for each $1\leq i\leq N$ denote $g_i:=g_{v_i}$. Let $C'$ be the constant from the lemma.

\begin{claim}
    There exists a permutation $\sigma:\{1,2,\ldots,N\}\rightarrow\{1,2,\ldots,N\}$ such that for each $1\leq k\leq N$ we have
    \[\Big|\sum_{i=1}^{k}\int_{\mathbb{R}^{2n}}g_{\sigma(i)}\omega^n\Big|\leq 2^{2n} C'\cdot\|f\|_{\infty}\varepsilon^{2n}.\]
    Moreover, there exists a Hamiltonian diffeomorphism $\Phi_{\lambda}\in\Ham_{\mathrm{c}}((-2,2)^{2n})$ such that
    \[(\forall\,1\leq i\leq N)(\forall\, x\in[-2\varepsilon/3,2\varepsilon/3]^{2n})\quad\Phi_{\lambda}(v_i+x)=v_{\sigma(i)}+x.\]
\end{claim}

\begin{proof}
    For each $ 1 \leq i \leq N $ let $s_i:=\int_{\mathbb{R}^{2n}}g_i\omega^n$, and note that $ |s_i| \leqslant 2^{2n} C'\cdot\|f\|_{\infty}\varepsilon^{2n} $. We will inductively construct $\sigma$. Start by defining $\sigma(1):=1$, and suppose we have constructed $\sigma(1),\ldots,\sigma(k)$. If $\sum_{i=1}^k s_{\sigma(i)}\geq 0$, we find an index $j\in\{1,\ldots,N\}\setminus\{\sigma(1),\ldots,\sigma(k)\}$ for which $s_j\leq 0$ and define $\sigma(k+1):=j$. Otherwise, we find an index $j$ in the same set for which $s_j\geq 0$ and again define $\sigma(k+1):=j$. In both cases such an index exists because $\sum_{i=1}^N s_i=0$.
    
    To construct the Hamiltonian diffeomorphism $\Phi_{\lambda}$, it is enough to show that we can do the transposition of two consecutive cubes $v_i+[-2\varepsilon/3,2\varepsilon/3]^{2n} $ and $v_{i+1}+[-2\varepsilon/3,2\varepsilon/3]^{2n} $ by a Hamiltonian diffeomorphism (then any permutation can be obtained as a product of such transpositions). To do that, we can simply find a Hamiltonian isotopy of cubes (which consists only of translations) which exchanges the two cubes, and then extend it to an ambient Hamiltonian isotopy. We leave details to the reader.
\end{proof}

\noindent Define $\widetilde{\Gamma}^{\varepsilon}_{\lambda}:=\{\frac{1}{2}v_i+\frac{1}{2}v_{i+1}\mid 1\leq i\leq N-1\}$, and for each $1\leq i\leq N-1$ let
\[v^-_{i+1}:=\frac{3}{4}v_{i+1}+\frac{1}{4}v_{i},\quad v^+_i:=\frac{3}{4}v_{i}+\frac{1}{4}v_{i+1},\]
\[F^{\pm}_{i}(x)=\frac{\mathcal{R}^{\varepsilon/6}(x-v^{\pm}_{i})}{(\varepsilon/6)^{2n}}.\]
For $ 1 \leq i \leq N $ denote $\widetilde{g}_i=g_{\sigma(i)}\circ\Phi_{\lambda}$, and define $a_i:=\sum_{j=1}^i\int_{\mathbb{R}^{2n}}\widetilde{g}_j\omega^n$. Furthermore, put $ a_0 =0 $, and also note that $ a_N = 0 $. Finally define
\[(\forall\, 1\leq i\leq N)\quad f_{\lambda,v_i}(x):=\widetilde{g}_{i}(x)+a_{i-1}\cdot F^{-}_{i}-a_i\cdot F^{+}_i,\]
\[(\forall\, 1\leq i\leq N-1)\quad f_{\lambda, \frac{v_i+v_{i+1}}{2}}(x):=a_i\cdot F^{+}_{i}-a_i\cdot F^{-}_{i+1}.\]
Note that $\sum_{v\in\Gamma^{\varepsilon}_{\lambda}\cup\widetilde{\Gamma}^{\varepsilon}_{\lambda}} f_{\lambda,v}=\sum_{i=1}^N \widetilde{g}_i=\Phi_{\lambda}^*f_{\lambda}$, and moreover
\[\int_{\mathbb{R}^{2n}}f_{\lambda, v_i}\omega^n= \int_{\mathbb{R}^{2n}} \widetilde{g}_i\omega^n+a_{i-1}-a_i=0,\quad \int_{\mathbb{R}^{2n}}f_{\lambda, \frac{v_i+v_{i+1}}{2}}\omega^n=a_i-a_i=0.\]
We also have $\|D^3 F^{\pm}_i\|_{\infty}=\frac{6^{2n+3}}{\varepsilon^{2n+3}}\|D^3\mathcal{R}\|_{\infty}+\mathcal{O}(\frac{1}{\varepsilon^{2n+2}})$, therefore it is enough to define $C:=C'\cdot(2+2^{2n} \cdot 6^{2n+3}\|D^3\mathcal{R}\|_{\infty})$ to get
\[(\forall\,v\in\Gamma^{\varepsilon}_{\lambda}\cup\widetilde{\Gamma}^{\varepsilon}_{\lambda})\quad\|D^3f_{\lambda,v}\|_{\infty}\leq\frac{C\|f\|_{\infty}}{\varepsilon^3}.\]
Finally, it is clear that for each $v\in\Gamma^{\varepsilon}_{\lambda}\cup\widetilde{\Gamma}^{\varepsilon}_{\lambda}$ we have $ \mathrm{supp}(f_{\lambda,v})\subset v+[-2\varepsilon/3,2\varepsilon/3]^{2n}$.
\end{proof}

For each $\lambda\in\mathcal{X}$ we will divide the family of functions $\{f_{\lambda,v}\mid v\in\Gamma^{\varepsilon}_{\lambda}\cup\widetilde{\Gamma}^{\varepsilon}_{\lambda}\}$ into a finite number of subsets (which does not depend on $\varepsilon$) such that the following holds: if $f_{\lambda,u}$ and $f_{\lambda,v}$ are in the same subset, then $u+(-4\varepsilon,4\varepsilon)^{2n}$ and $v+(-4\varepsilon,4\varepsilon)^{2n}$ are disjoint. This allows us to apply Theorem \ref{ThmMicro} (the microscale statement) to the sum of the functions in the same subset.\\

\noindent For each $\mu\in\mathcal{Y}:=\{0,1,\ldots,9\}^{2n}$ define the finite grid $\mathcal{G}^{\varepsilon}_{\lambda,\mu}\subset (-1,1)^{2n}$ as
\[\mathcal{G}^{\varepsilon}_{\lambda,\mu}:=(\varepsilon\mu+10\varepsilon\mathbb{Z}^{2n})\cap(\Gamma^{\varepsilon}_{\lambda}\cup\widetilde{\Gamma}^{\varepsilon}_{\lambda}).\]

\begin{claim}
    For each $\mu\in\mathcal{Y}$, there exists $\Psi_{\mu}\in\Ham_{\mathrm{c}}((-8,8)^{2n})$ such that 
    \[(\forall v\in\mathcal{G}^{\varepsilon}_{\mu})(\exists c)\quad\Psi_{\mu}^*x_1-x_1=\frac{1}{\varepsilon}\cdot x_2+c\text{ holds on }v+(-4\varepsilon,4\varepsilon)^{2n}.\]
\end{claim}

\begin{proof}
    Note that for any function $h\in C^{\infty}(\mathbb{R})$ the diffeomorphism
    \[\varphi_{h}(x_1,x_2,\ldots,x_{2n}):= (x_1+h(x_2),x_2,\ldots,x_{2n})\]
    is Hamiltonian and satisfies $\varphi_{h}^*x_1-x_1=h(x_2)$. Writing $\mu=(\mu_1,\mu_2,\ldots,\mu_{2n})$, choose $h:\mathbb{R}\rightarrow (-5,5) $ to be a smooth function that satisfies
   \[h'(t)=\frac{1}{\varepsilon} \,\, \text{ for } \,\, t\in \varepsilon \mu_2+10\varepsilon\mathbb{Z}+[-4\varepsilon,4\varepsilon].\]
    Let $H$ be the Hamiltonian function which generates the Hamiltonian isotopy $t\mapsto\varphi_{th}$. Let $\chi\in C^{\infty}_{\mathrm{c}}((-8,8)^{2n})$ be a cut-off function which equals $1$ on $ [-7,7]^{2n}$. Then we can define $\Psi_{\mu}$ as the time-1 map of the Hamiltonian isotopy generated by the Hamiltonian function $\chi\cdot H$.
\end{proof}

\noindent Let $\widetilde{C}$ be the constant from the statement of Theorem \ref{ThmMicro}, and let $m=m(n)\in\mathbb{N}$ be a natural number for which $\frac{C}{m}\leq\widetilde{C}$. 
Note that for each $v\in\Gamma^{\varepsilon}_{\lambda}\cup\widetilde{\Gamma}^{\varepsilon}_{\lambda}$, the function $f_{\lambda,v}$ is supported in $v+(-\varepsilon, \varepsilon)^{2n}$, therefore we can apply Theorem \ref{ThmMicro} to the function $\frac{1}{m} f_{\lambda,v}$ to get Hamiltonian diffeomorphisms $\{\Phi^{\lambda,v}_{i,\pm}\}_{i=1}^{4n}$ supported in $v+(-4\varepsilon, 4\varepsilon)^{2n}$ such that
\[ \frac{1}{m} f_{\lambda,v}=\sum_{i=1}^{4n}(\Phi^{\lambda,v}_{i,+})^*\Big(\Psi_{\mu}^*x_1-x_1\Big)-(\Phi^{\lambda,v}_{i,-})^*\Big(\Psi_{\mu}^*x_1-x_1\Big).\]
Moreover, for each $ \mu \in \mathcal{Y} $, for any two distinct elements $u,v\in\mathcal{G}^{\varepsilon}_{\lambda,\mu}$ the supports of $\Phi^{\lambda,u}_{i,\pm}$ and $\Phi^{\lambda,v}_{j,\pm}$ are disjoint, which implies that if for every $ \mu \in \mathcal{Y} $ and $ 1 \leqslant i \leqslant 4n $ we denote by $\{\Psi^{\lambda,\mu}_{i,+}\} $ the composition of all $\Phi^{\lambda,u}_{i,+}$ (when $ u \in \mathcal{G}^{\varepsilon}_{\lambda,\mu}$), and by $\{\Psi^{\lambda,\mu}_{i,-}\} $ the composition of all $\Phi^{\lambda,u}_{i,-}$ (when $ u \in \mathcal{G}^{\varepsilon}_{\lambda,\mu}$), then we get
\[\frac{1}{m} \Phi^*_{\lambda}f_{\lambda}= \frac{1}{m} \sum_{v\in\Gamma^{\varepsilon}_{\lambda}\cup\widetilde{\Gamma}^{\varepsilon}_{\lambda}} f_{\lambda,v} = 
\sum_{\mu\in\mathcal{Y}}\sum_{i=1}^{4n} (\Psi^{\lambda,\mu}_{i,+})^*\Big(\Psi_{\mu}^*x_1-x_1\Big)-(\Psi^{\lambda,\mu}_{i,-})^*\Big(\Psi_{\mu}^*x_1-x_1\Big).\]
Finally, this means that 
$$ f=\sum_{\lambda\in\mathcal{X}} f_{\lambda} =  m \sum_{\lambda\in\mathcal{X}} \sum_{\mu\in\mathcal{Y}}\sum_{i=1}^{4n}  (\Phi^{-1}_{\lambda})^* (\Psi^{\lambda,\mu}_{i,+})^*\Big(\Psi_{\mu}^*x_1-x_1\Big) - (\Phi^{-1}_{\lambda})^* (\Psi^{\lambda,\mu}_{i,-})^*\Big(\Psi_{\mu}^*x_1-x_1\Big), $$
which gives us a representation $(\ref{eq:thmlocal})$ for $ f $, with $ N = 8nm  |\mathcal{X}|  |\mathcal{Y}| = 2^{2n+3} 10^{2n} nm. $

\section{Proof of Theorem \ref{ThmMain}} \label{sec:proofThmMain}

Proposition \ref{MainProp2} from Section \ref{Section:Microscale} admits the following corollary:
\begin{cor} \label{cor:coord-change}
Let $ U \subset \mathbb{R}^{2n} $ be an open set, let $ f : U \rightarrow \mathbb{R} $ be a smooth function, and let $ p \in U $ such that $ df(p) = dx_1 $.
Then for any neigbourhood $ p \in U' \subset U $ there exists a Hamiltonian diffeomorphism $ \Theta $ compactly supported in $ U' $ such that
$ \Theta (p) = p $ and such that $ f \circ \Theta(x) = x_1 + c $ on a neighbourhood of $ p $, where $ c \in \mathbb{R} $ is a constant.
\end{cor}
\begin{proof}
Without loss of generality we may assume that $ p = 0 $ and $ f(p) = f(0) = 0 $. Pick a smooth function $ \rho : \mathbb{R}^{2n} \rightarrow \mathbb{R} $ supported in $ (-2,2)^{2n} $, such that $ \rho = 1 $ on $ [-1,1]^{2n} $. 
For $ \delta > 0 $ define $ h_\delta(x) = \rho(x/\delta)(f(x) - x_1) $. If we choose $ \delta $ to be small enough, we get $ (-4\delta,4\delta)^{2n} \subset U' \subset U $ and $ \| Dh_\delta \|_\infty < 1/8 $, and then by applying Proposition \ref{MainProp2} with $ L = 2\delta $ and $ \lambda = 1 $ we get a Hamiltonian diffeomorphism $\Phi\in\Ham_{\mathrm{c}}((-4\delta,4\delta)^{2n})$ such that $ \Phi^* x_1-x_1 = h_\delta $ on $ (-2\delta,2\delta)^{2n} $, in particular $  \Phi^* x_1 = x_1 + h_\delta = f $ on $ (-\delta,\delta)^{2n} $. Finally, if $ \Phi' \in \Ham_{\mathrm{c}}((-4\delta,4\delta)^{2n}) $ is a Hamiltonian diffeomorphism such that $ \Phi'(x) = x - \Phi(0) $ on a neighbourhood of $ \Phi(0) $, then for $ \Theta := (\Phi' \circ \Phi)^{-1} $ we have $ \Theta (0) = 0 $ and $ f \circ \Theta  = x_1 $ on a neighbourhood of $ 0 $.
\end{proof}

The next two lemmas will help us to reduce Theorem \ref{ThmMain} to the local statement (Theorem \ref{ThmLocal}).

\begin{lemma}\label{extended-Darboux_nbhd-Lemma}
Let $(M,\omega)$ be closed and connected symplectic manifold, and let $u\in C_0^{\infty}(M)$ be a non-zero smooth function. There exists a Darboux chart $ (V,\varphi) $ and a Hamiltonian diffeomorphism $\Psi \in \Ham(M,\omega)$ such that $ 0 \in \varphi(V) $ and such that we have 
$$ u \circ \varphi^{-1} (x_1,\ldots,x_{2n}) = x_1 + c $$ and 
$$ u\circ \Psi \circ \varphi^{-1} (x_1,\ldots,x_{2n}) = -x_1 - c $$
on $ \varphi(V) \subset\mathbb{R}^{2n}$, where $ c \in \mathbb{R} $ is some constant. 
\end{lemma}
\begin{proof}
 Since $u$ is non-zero and normalized, by Sard's theorem there exist points $p, q \in M$ such that $du(p)\neq 0$ and $ du(q)\neq 0 $, and moreover $ u(q)=-u(p) $. Pick some $ \Psi \in \Ham(M,\omega) $ such that $ \Psi(p) = q $. Choose a Darboux chart $ (V,\varphi) $ around the point $ p $, such that $ \varphi(p) = 0 \in \mathbb{R}^{2n} $. By modifying $ (V,\varphi) $ and $ \Psi $, if necessary, we may without loss of generality assume that $ d(u \circ \varphi^{-1})(0) = dx_1 $ and $ d(u\circ \Psi \circ \varphi^{-1})(0) = -dx_1 $. By Corollary \ref{cor:coord-change} we can find Hamiltonian diffeomorphisms $ \Theta_+,\Theta_- $ compactly supported in $ \varphi(V) $ such that $ \Theta_+(0) = \Theta_-(0) = 0 $ and such that $ u \circ \varphi^{-1} \circ \Theta_+(x_1,\ldots,x_{2n}) = x_1 + c $ and $ u \circ \Psi \circ \varphi^{-1} \circ \Theta_-(x_1,\ldots,x_{2n}) = -x_1 - c $ on a neighbourhood of $ 0 \in \mathbb{R}^{2n} $. Replacing $ \varphi $ by $ \Theta_+^{-1} \circ \varphi $ and $ \Psi $ by $\Psi\circ\varphi^{-1}\circ\Theta_{-}\circ\Theta_{+}^{-1}\circ\varphi$, and then shrinking $ V $ while still keeping that $ p \in V $, we get that 
$$ u \circ \varphi^{-1} (x_1,\ldots,x_{2n}) = x_1 + c $$ 
$$ u \circ \Psi \circ \varphi^{-1} (x_1,\ldots,x_{2n}) = -x_1 - c $$ 
hold on $ \varphi(V) $, where $ c =u(p) $. 
\end{proof}

\begin{lemma}[Localization]\label{localizationLemma}
    Let $M$ be a closed and connected manifold endowed with a volume form $ \Omega $, and let $\mathcal{U}=\{U_1,\ldots, U_m\}$ be an open cover of $M$. Then for every smooth function $f:M\rightarrow\mathbb{R}$ with $\int_M f\Omega=0$ one can find smooth functions $f_1,\ldots,f_m:M\rightarrow\mathbb{R}$ such that:
    \begin{enumerate}[label=(\roman*)]
        \item $f=\sum_{i=1}^m f_i$.
        \item For each $ i $, $f_i$ is supported inside $U_i$.
        \item $\int_{U_i}f_i\Omega=0$.
        \item $(\exists\, C=C(\mathcal{U}))\,\|f_i\|_{\infty}\leq C\|f\|_{\infty}$ for every $ i $.
    \end{enumerate}
\end{lemma}

\begin{proof}
 The proof is by induction on $ m $. The case $ m = 1 $ is clear, so assume that $ m \geqslant 2 $. Since $ M $ is connected, there exist distinct $ i,j $ such that $ U_i \cap U_j \neq \emptyset $. Without loss of generality we may assume that $ i=m-1 $ and $ j=m $. Choose an open set $ V \Subset U_{m-1} \cup U_m $ such that $ U_1, \ldots, U_{m-2}, V $ cover $ M $ (it exists by compactness). Then, choose smooth functions $ \rho, \rho' \in C^\infty(M) $ such that $ \rho $ is supported in $ U_{m-1} $, $ \rho' $ is supported in $ U_m $, and $ \rho + \rho' = 1 $ on $ V $. Also, since $ U_{m-1} \cap U_m \neq \emptyset $, we can find a smooth function $ h \in C^\infty(M) $ supported in $ U_{m-1} \cap U_m $ such that $ \int_M h \Omega = 1 $. 
 
 Now assume that we have some smooth function $f:M\rightarrow\mathbb{R}$ with $\int_M f\Omega=0$. Applying the induction hypothesis to the cover $ U_1, \ldots, U_{m-2}, V $ of $ M $, we obtain smooth functions $ f_1, \ldots, f_{m-2}, g $ such that $ f = f_1 + \cdots + f_{m-2} + g $, for $ 1 \leqslant i \leqslant m-2 $ the function $ f_i $ is supported in $ U_i $, the function $ g $ is supported in $ V $, we have 
 $$ \int_{U_1} f_1 \Omega = \cdots =  \int_{U_{m-2}} f_{m-2} \Omega = \int_{V} g \Omega = 0, $$
and $ \|f_i\|_{\infty}, \| g \|_\infty \leq C\|f\|_{\infty}$. Denote $ \lambda = \int_M \rho g \Omega $, and then define 
$$ f_{m-1} = \rho g - \lambda h ,$$
$$ f_m = \rho' g + \lambda h .$$
It is clear that the functions $ f_1, \ldots, f_m $ satisfy the needed requirements.
\end{proof}

The remaining ingredient needed for the proof of Theorem \ref{ThmMain} is an averaging property which was introduced by Polterovich in \cite{P}:

\begin{theorem}[Averaging property] \label{thm:AvProp}
Let $ (M,\omega) $ be a connected symplectic manifold, and let $ H : M \rightarrow \mathbb{R} $ be a compactly supported continuous function such that $ \int_M H \omega^n = 0 $. Then for every $ \varepsilon > 0 $ there exist $ \Theta_1, \ldots, \Theta_\ell \in \Ham(M,\omega) $ such that 
$$ \frac{1}{\ell} | H \circ \Theta_1(x) + \cdots + H \circ \Theta_\ell(x) | < \varepsilon $$
for every $ x \in M $.
\end{theorem}
This statement was shown in \cite{P} in the case when $ M $ is closed. We believe that the arguments of \cite{P} can be adapted also for showing Theorem \ref{thm:AvProp}. For convenience of the reader, in Section \ref{sec:AvProp} below we give a sketch of an alternative argument.

Now we turn to the proof of Theorem \ref{ThmMain}. Let $u\in C_0^{\infty}(M)$ be a non-zero smooth and normalized function. First apply Lemma \ref{extended-Darboux_nbhd-Lemma} to obtain a Darboux chart $ (V,\varphi) $ and a Hamiltonian diffeomorphism $\Psi \in \Ham(M,\omega)$ such that $ 0 \in \varphi(V) $ and such that we have
$$ u \circ \varphi^{-1} (x_1,\ldots,x_{2n}) = x_1 + c ,$$  
$$ u\circ \Psi \circ \varphi^{-1} (x_1,\ldots,x_{2n}) = -x_1 - c $$
hold on $ \varphi(V) \subset\mathbb{R}^{2n}$. Choose $ L > 0 $ such that $[-8L,8L]^{2n} \subset \varphi(V) $, and denote $ W = \varphi^{-1}((-L,L)^{2n}) $. Recall that $\varphi(p)=0$, and for every point $ q \in M $ choose $\Psi_q\in\Ham(M,\omega)$ such that $\Psi_q(p)=q$, denote $ W_q = \Psi_q(W) $ and $ V_q = \Psi_q(V) $, and define Darboux coordinates $ \varphi_q:  V_q \rightarrow \mathbb{R}^{2n} $ by $ \varphi_q = \varphi \circ \Psi_q^{-1} $. We obtain an open cover $ \cup_{q \in M} W_q = M $, and by compactness we can choose a finite subcover $ \cup_{j=1}^m W_{q_j} = M $. Let $ C_{1} = C( \{ W_{q_j} \} )  > 0 $ be the constant given by Lemma \ref{localizationLemma} when applied to the cover $ \{ W_{q_j} \} $ of $ M $. To simplify the notation, denote $ W_j := W_{q_j} $, $ \Psi_j := \Psi_{q_j} $ and $ \varphi_{j} := \varphi_{q_j} $ for $ 1 \leq j \leq m $.

Define $ v := u + u \circ \Psi \in C_0^\infty(M) $. The function $ v $ is normalized and its support lies in $ \widetilde{M} := M \setminus \varphi^{-1}([-8L,8L]^{2n}) $. Applying Theorem \ref{thm:AvProp} to the restriction of the function $ v $ to $ \widetilde{M} $, we obtain $ \Theta_1, \ldots, \Theta_\ell \in \Ham_{\mathrm{c}}(\widetilde{M},\omega) \subset \Ham(M,\omega)$ such that
$$ \frac{1}{\ell} | v \circ \Theta_1 + \cdots + v \circ \Theta_\ell | < \frac{L}{2m C_1} $$
holds on $ M $. Define smooth normalized functions
$$ w_+ = u \circ \Theta_1 + \cdots + u \circ \Theta_\ell ,$$
$$ w_- = u \circ \Psi \circ \Theta_1 + \cdots + u \circ \Psi \circ \Theta_\ell ,$$
$$ w = w_+ + w_- = v \circ \Theta_1 + \cdots + v \circ \Theta_\ell ,$$
and note that we have $ w_+ \circ \varphi^{-1} (x) = \ell x_1 + \ell c $ and $ w_- \circ \varphi^{-1} (x) = -\ell x_1 - \ell c $ for $ x \in [-8L,8L]^{2n} $, and therefore
for each $ 1 \leq i \leq m $ we have 
\begin{equation} \label{eq:w-std}
\begin{gathered}
 w_+ \circ \Psi_i^{-1} \circ \varphi_i^{-1} (x) = \ell x_1 + \ell c, \\
 w_- \circ \Psi_i^{-1} \circ \varphi_i^{-1} (x) = -\ell x_1 - \ell c
\end{gathered}
\end{equation}
for every $ x \in [-8L,8L]^{2n} $.

Now let $ f \in C_0^\infty(M) $ be a smooth normalized function with $ \| f \|_\infty \leq 1 $. Define $ h \in C_0^\infty(M) $ by 
$$ h = \frac{1}{N_1} f - w \circ \Psi_1^{-1} - \cdots - w \circ \Psi_m^{-1} ,$$
where $ N_1 = \lceil \frac{2C_1}{\ell L} \rceil \in \mathbb{N} $. Then we have
$$ \| h \|_\infty \leq \frac{1}{N_1} \| f \|_\infty + m \| w \|_\infty \leq  \frac{ \ell L}{2C_1} + \frac{\ell L}{2C_1} = \frac{\ell L}{C_1}. $$
By Lemma \ref{localizationLemma} we can decompose $ h = h_1 + \cdots + h_m $, where for each $ 1 \leq i \leq m $, $ h_i \in C_0^\infty(M) $ is supported in $ W_i $, and 
$ \| h_i \|_\infty \leq \ell L $. Then for every $ i $, because of $(\ref{eq:w-std})$ we can apply Theorem \ref{ThmLocal} (the local statement) to $ \frac{1}{\ell } h_i \circ \varphi_i^{-1} $ to obtain $ \Phi_{i,j}^+ , \Phi_{i,j}^- \in \Ham(M,\omega) $ ($ 1 \leq j \leq N $) supported in $ \varphi_i^{-1} ((-8L,8L)^{2n}) $ such that the equality
$$ h_i =  \sum_{j=1}^N  (\Phi_{i,j}^+)^* (\Psi_i^{-1})^* w_+ +  \sum_{j=1}^N  (\Phi_{i,j}^-)^* (\Psi_i^{-1})^* w_- . $$
holds on $ \varphi_i^{-1} ((-8L,8L)^{2n}) $. But this means that the equality 
\begin{equation} \label{eq:h-i-rep}
 h_i = - w \circ \Psi_i^{-1} + \sum_{j=1}^N  (\Phi_{i,j}^+)^* (\Psi_i^{-1})^* w_+ +  \sum_{j=1}^N  (\Phi_{i,j}^-)^* (\Psi_i^{-1})^* w_- 
\end{equation} 
holds both on $ \varphi_i^{-1} ((-8L,8L)^{2n}) $ (since $ w \circ \Psi_i^{-1} $ vanishes there) and on $ M \setminus \varphi_i^{-1} ((-8L,8L)^{2n}) $ (since both the left-hand side and
the right-hand side vanish there). That is, $(\ref{eq:h-i-rep})$ holds on $ M $.
Summarizing, we get
\begin{equation*}
\begin{gathered}
f = N_1 h + N_1 \sum_{i=1}^m w \circ \Psi_i^{-1} =  N_1 \sum_{i=1}^m h_i + N_1 \sum_{i=1}^m w \circ \Psi_i^{-1} \\
=  N_1  \sum_{i=1}^m \sum_{j=1}^N  (\Phi_{i,j}^+)^* (\Psi_i^{-1})^* w_+ +  N_1 \sum_{i=1}^m \sum_{j=1}^N  (\Phi_{i,j}^-)^* (\Psi_i^{-1})^* w_- \\
=  N_1  \sum_{i=1}^m \sum_{j=1}^N  \sum_{k=1}^\ell (\Phi_{i,j}^+)^* (\Psi_i^{-1})^* \Theta_k^* u +  N_1 \sum_{i=1}^m \sum_{j=1}^N  \sum_{k=1}^\ell (\Phi_{i,j}^-)^* (\Psi_i^{-1})^* 
\Theta_k^* \Psi^* u
\end{gathered}
\end{equation*}
and this finishes the proof of Theorem \ref{ThmMain}.

\subsection{The averaging property} \label{sec:AvProp}

\begin{proof}[Proof of Theorem \ref{thm:AvProp}]
    Denote $\mathcal{D}_0:=C^0_c(M)$, with its usual inductive limit topology induced by the sup norms on compact subsets of $M$. Let $\mathcal{D}_0'$ be the dual space, with the topology induced by $|\langle\cdot,u\rangle|,\,u\in\mathcal{D}_0$. Let $\mathcal{V}\subset C(M)$ be the subspace of functions vanishing at infinity, endowed with the sup norm. Let $\operatorname{conv}_0(H)\subset\mathcal{V}$ be the closed (in the $C^0$ topology) convex hull of the $\Ham_{\mathrm{c}}(M,\omega)$ orbit of $H$. It suffices to show that $0\in \operatorname{conv}_0(H)$. Indeed, if $0\in \operatorname{conv}_0(H)$, then for any $\varepsilon>0$ there exists $l\in\mathbb{N}$, $\alpha_1,\ldots,\alpha_l\in [0,1]$ with $\alpha_1+\cdots+\alpha_l=1$, and $\Theta_1,\ldots,\Theta_l\in\Ham_{\mathrm{c}}(M,\omega)$ such that
    \[\forall x\in M,\quad\quad\Big|\sum_{j=1}^{l}\alpha_j\cdot H(\Theta_j(x))\Big|<\varepsilon/2.\]
    Next, we find $\beta_1,\ldots,\beta_{l}\in\mathbb{Q}\cap[0,1]$, with $\beta_1+\cdots+\beta_l=1$, such that 
    \[\sum_{j=1}^{l}|\alpha_j-\beta_j|<\frac{\varepsilon}{2\cdot\max_{x\in M}|H(x)|}.\]
    From here we get 
    \[\forall x\in M,\quad\quad\Big|\sum_{j=1}^{l}\beta_j\cdot H(\Theta_j(x))\Big|<\varepsilon,\]
    which implies the averaging property.\\

    \noindent We argue by contradiction that $0\in \operatorname{conv}_0(H)$. Assume that $0\not\in \operatorname{conv}_0(H)$. Then one can find an element $h\in\mathcal{V}'$ in the dual space of $\mathcal{V}$, such that $\langle h,K\rangle\geq 1$ for all $K$ in the $\Ham_{\mathrm{c}}(M,\omega)$ orbit of $H$. The group $\Ham_{\mathrm{c}}(M,\omega)$ acts on $\mathcal{V}$ by pull back and on $\mathcal{V}'$ by push forward, therefore
    \begin{equation}\label{equationI}
        \forall\Phi,\Psi\in\Ham_{\mathrm{c}}(M,\omega)\quad\langle\Phi_*h,\Psi^* H\rangle=\langle  h,\Phi^*\Psi^* H\rangle\geq 1.
    \end{equation}
    Note that $\mathcal{V}'\subset\mathcal{D}_0'$. It follows from (\ref{equationI}) that $\langle h',K\rangle\geq 1$ for any $K$ in the convex hull of the $\Ham_{\mathrm{c}}(M,\omega)$ orbit of $H$, and for any $h'$ in the $\mathcal{D}_0'$ closure of the convex hull of the $\Ham_{\mathrm{c}}(M,\omega)$ orbit of $h$.

    \begin{lemma}[Lemma 2.3 in \cite{Le} for distributions of finite order]
        For any $h\in \mathcal{D}_0'$, smooth distributions are dense in the closed convex hull of $\Ham_{\mathrm{c}}(M,\omega)$ orbit of $h$.
    \end{lemma}
    \noindent It follows that we can choose a smooth distribution $h'$, given by integration against a smooth function $\phi$ on $M$. Now we get
    \begin{equation}\label{equationII}
        \langle h',K\rangle=\int_M \phi K\omega^n\geq 1,
    \end{equation}
    for all $K$ in the convex hull of the $\Ham_{\mathrm{c}}(M,\omega)$ orbit of $H$.

    \begin{lemma}[Lemma 4.5 in \cite{Le}]\label{Lemma4.5}

        If $f\in L^1(X)$ and $E\subset X$ has a positive measure, then the function
        $$
        f'=\begin{cases}
			\fint_E f, & \text{on $E$}\\
            f, & \text{on $X\setminus E$}
		 \end{cases}
        $$
        is in the closure, in the $L^1(X)$ topology, of the convex hull of the $\Ham_{\mathrm{c}}(X)$ orbit of $f$.
    \end{lemma}
    
    \noindent Pick a relatively compact, connected, open neighbourhood $X\subset M$ of $\mathrm{supp}\,H$, and apply Lemma \ref{Lemma4.5} with $E=X$ and $f=H$. We get functions $H_j$ in the convex hull of the $\Ham_{\mathrm{c}}(M,\omega)$ orbit of $H$, that converge in $L^1$ to $0$, which contradicts (\ref{equationII}).
\end{proof}

\section{Proof of Theorem \ref{FiniteDimThm}}\label{FinieDimSection}

We begin with the following two claims.

\begin{claim}\label{RepresentingOpenSet}
    For every non-zero element $\xi\in\mathfrak{g}$, one can find a smooth curve $\gamma_i:\mathbb{R}\rightarrow G$ for each $1\leq i\leq n=\mathrm{dim}\,\mathfrak{g}$, such that the image of the function
    \[f:\mathbb{R}^n\rightarrow\mathfrak{g},\quad f(t_1,\ldots,t_n)=\sum_{i=1}^n\mathrm{Ad}_{\gamma_i(t_i)}\xi\]
    contains an open subset of $\mathfrak{g}$.
\end{claim}

\begin{proof}
    Since the Lie algebra $\mathfrak{g}$ is simple, the ideal generated by $\xi$ has to be the entire $\mathfrak{g}$. In particular, we have
    \begin{equation}\label{LieSpan}
        \mathrm{span}\{[[[\xi,\eta_1],\eta_2],\ldots,\eta_m]\mid m\in\mathbb{N},\eta_1,\ldots,\eta_m\in\mathfrak{g}\}=\mathfrak{g}.
    \end{equation}
    Consider the family of functions
    \[\mathcal{F}=\left\{\begin{array}{l}
    F:\mathbb{R}^m\rightarrow\mathfrak{g},\quad F(t_1,t_2,\ldots,t_m)= \\
    =\frac{\partial}{\partial t_m}(\mathrm{Ad}_{\mathrm{exp}(t_m\eta_m)}\cdots\mathrm{Ad}_{\mathrm{exp}(t_2\eta_2)} \mathrm{Ad}_{\mathrm{exp}(t_1\eta_1)}\xi)
    \end{array}
    \Big\vert\, \begin{array}{l}
    m\in\mathbb{N}, \\
    \eta_1,\ldots,\eta_m\in\mathfrak{g}
  \end{array}\right\}.\]
    Using the identity $\mathrm{Ad}_{\mathrm{exp}(t\eta)}\xi=\xi+t[\xi,\eta]+\mathcal{O}(t^2)$ we get
    \begin{equation}\label{DiffSpan}
        \frac{\partial^m}{\partial t_1\partial t_2\cdots \partial t_m}\Big\vert_{t=0}\mathrm{Ad}_{\mathrm{exp}(t_m\eta_m)}\cdots\mathrm{Ad}_{\mathrm{exp}(t_2\eta_2)} \mathrm{Ad}_{\mathrm{exp}(t_1\eta_1)}\xi=[[[\xi,\eta_1],\eta_2],\ldots,\eta_m].
    \end{equation}
    We claim that
    \begin{equation}\label{ImagesSpan}
        \mathrm{Span}\bigcup_{F\in\mathcal{F}}\mathrm{Im}\,F=\mathfrak{g}.
    \end{equation}
    To prove our claim we argue by contradiction. Assume that (\ref{ImagesSpan}) does not hold. Then we can find a non-trivial linear functional $\varphi:\mathfrak{g}\rightarrow\mathbb{R}$ such that $\varphi|_{\mathrm{Im}F}=0$ for all $F\in\mathcal{F}$. In particular, we have $\varphi(\frac{\partial F}{\partial t_m}(t_1,\ldots,t_m))=0$ for all $F\in\mathcal{F}$ and $(t_1,\ldots,t_m)\in\mathbb{R}^m$. By differentiating last equation at $0$ we get $\varphi(\frac{\partial^m F}{\partial t_1\cdots\partial t_m}(0))=0$ for all $F\in\mathcal{F}$. Finally, using (\ref{LieSpan}) and (\ref{DiffSpan}) we get that $\varphi$ is identically equal to $0$, contradicting the fact that $\varphi$ is non-trivial.\\

    Let $n$ be the dimension of $\mathfrak{g}$. Equation (\ref{ImagesSpan}) implies that we can find a basis for $\mathfrak{g}$ of the form $\mathcal{B}=\{F_i(t^i_1,\ldots,t^i_{m_i})\mid 1\leq i\leq n,\,F_i\in\mathcal{F}\}$. Write each element of the basis $\mathcal{B}$ as 
    \[F_i(t^i_1,\ldots,t^i_{m_i}) = \frac{\partial}{\partial t_m}(\mathrm{Ad}_{\mathrm{exp}(t_m\eta^i_{m_i})}\cdots\mathrm{Ad}_{\mathrm{exp}(t_2\eta^i_2)} \mathrm{Ad}_{\mathrm{exp}(t_1\eta^i_1)}\xi)=\frac{d}{ds}\Big\vert_{s=0}\mathrm{Ad}_{\gamma_i(s)}\xi,\]
    where $\gamma_i:\mathbb{R}\rightarrow G$ is defined as 
    \[\gamma_i(s)=\mathrm{exp}((t^i_{m_i}+s)\eta^i_{m_i})\mathrm{exp}(t^i_{m_i-1}\eta^i_{m_i-1})\cdots\mathrm{exp}(t^i_1\eta^i_{1}).\]
    The fact that $\mathrm{Span}\{\frac{d}{ds}|_{s=0}\mathrm{Ad}_{\gamma_i(s)}\xi\mid 1\leq i\leq n\}=\mathfrak{g}$ implies that the function $f:\mathbb{R}^n\rightarrow\mathfrak{g}$ defined as
    \[f(t_1,\ldots,t_n)=\sum_{i=1}^n\mathrm{Ad}_{\gamma_i(t_i)}\xi\]
    has an invertible differential at the origin. Finally, the implicit function theorem implies that the image of $f$ contains an open neighbourhood of $f(0)$ in $\mathfrak{g}$.
\end{proof}

\begin{claim}\label{ConvexCombination}
    For all $\omega\in\mathfrak{g}$ there exists $g_1,g_2,\ldots,g_{n+1}\in G$ and $\lambda_1,\ldots,\lambda_{n+1}\in[0,1]$ such that $\sum_{i=1}^{n+1}\lambda_i=1$ and $\sum_{i=1}^{n+1}\lambda_i\cdot\mathrm{Ad}_{g_i}\omega=0$.
\end{claim}

\begin{proof}
   Let $d\mu$ be a Haar measure on $G$, and let $\eta=\int_{G}\mathrm{Ad}_{g}\omega\, d\mu(g)$. Note that
    \[\forall h\in G\quad\mathrm{Ad}_{h}\eta=\int_{G}\mathrm{Ad}_{hg}\omega \,d\mu(g)=\int_{G}\mathrm{Ad}_{hg}\omega\, d\mu(hg)=\eta.\]
    The above equation implies that $\eta=0$, and hence we have
    \begin{equation}\label{AverageAction}
        \int_{G}\mathrm{Ad}_g\omega\,d\mu(g)=0.
    \end{equation}
    Therefore $0$ is in the convex hull of the adjoint orbit of $\omega$ (a compact set). The claim now follows from Carath\'eodory's convex hull theorem.
\end{proof}

\noindent Claim \ref{RepresentingOpenSet} implies that there exists an open subset $\mathcal{U}\subset\mathfrak{g}$ with $\mathcal{U}\subset\mathcal{V}_{\xi,n}$. Choose $0\neq\omega\in\mathcal{U}$ and apply Claim \ref{ConvexCombination} to get $m\leq n+1$, $g_1,\ldots,g_{m}\in G$ and $\lambda_1,\ldots,\lambda_{m}\in(0,1]$ with $\sum_{i=1}^{m}\lambda_i\cdot\mathrm{Ad}_{g_i}\omega=0$. For each $1\leq i\leq m$ define open set $\mathcal{U}_i=\{\mathrm{Ad}_{g_i}\eta\mid\eta\in\mathcal{U}\}\subset\mathfrak{g}$. Since $\mathrm{Ad}_g\mathrm{Ad}_h=\mathrm{Ad}_{gh}$ we have $\mathcal{U}_i\subset\mathcal{V}_{\xi,n}$. Let $r_1,\ldots,r_{m}\in\mathbb{Q}\cap(0,1]$ be a rational approximation of $\lambda_1,\ldots,\lambda_m$, such that $\frac{\lambda_i}{r_i}\cdot\mathrm{Ad}_{g_i}\omega\in\mathcal{U}_i$. Write $r_i=\frac{k_i}{N}$, where $k_i\in\mathbb{N}$ and $N \in \mathbb{N} $, then we have
\[0=N\cdot\sum_{i=1}^{m}\lambda_i\cdot\mathrm{Ad}_{g_i}\omega=\sum_{i=1}^{m}k_i\cdot\left(\frac{\lambda_i}{r_i}\cdot\mathrm{Ad}_{g_i}\omega\right).\]
This equality implies that $0$ belongs to the open set
\[\underbrace{\mathcal{U}_{1}+\ldots+\mathcal{U}_{1}}_{k_1}+\underbrace{\mathcal{U}_{2}+\ldots+\mathcal{U}_{2}}_{k_2}+\ldots+\underbrace{\mathcal{U}_{m}+\ldots+\mathcal{U}_{m}}_{k_m}.\]
Therefore $\mathcal{V}_{\xi,\sum_{i=1}^mnk_i}$ contains an open neighbourhood of $0$.\qed

\end{document}